\documentclass{amsart}
\usepackage{amssymb,euscript,amsmath, mathrsfs}
\usepackage[dvips]{graphicx}
\usepackage[dvips]{color}
\usepackage{epsfig}

\newcounter{ENUM}
\newcommand{\itm}{\item}
\newenvironment{ilist}{\renewcommand{\theENUM}{\roman{ENUM}}\renewcommand{\itm}{\addtocounter{ENUM}{1}\item[(\theENUM)]}\begin{itemize}\setcounter{ENUM}{0}}{\end{itemize}}
\newenvironment{alist}[1][0]{\renewcommand{\theENUM}{\alph{ENUM}}\renewcommand{\itm}{\addtocounter{ENUM}{1}\item[\theENUM)]}\begin{itemize}\setcounter{ENUM}{#1}}{\end{itemize}}

\def\b{{\mathbf b}}
\def\bc{{\mathbf c}}
\def\be{{\mathbf e}}
\def\bf{{\mathbf f}}

\def\1{{\mathbf 1}}
\def\n{{\mathbf n}}
\def\bp{{\mathbf p}}
\def\br{{\mathbf r}}

\def\x{{\mathbf x}}
\def\y{{\mathbf y}}
\def\z{{\mathbf z}}

\def\Z{{\mathbb Z}}

\def\R{{\mathbb R}}
\def\C{{\mathbb C}}

\def\cN{{\mathcal N}}

\def\cR{{\mathcal R}}
\def\cT{{\mathcal T}}
\def\M{{\mathcal M}}

\def\S{{\mathfrak S}}

\def\aux{\operatorname{aux}}

\def\Span{\mathrm{span}}
\def\vol{\mathrm{Vol}}

\def\vert{\mathrm{Vert}}

\def\tcone{\mathrm{tcone}}
\def\fcone{\mathrm{fcone}}
\def\Cone{\mathrm{Cone}}

\newtheorem{thm}{Theorem}[section]

\newtheorem{lem}[thm]{Lemma}
\newtheorem{cor}[thm]{Corollary}

\theoremstyle{definition}
\newtheorem{defn}[thm]{Definition}

\newtheorem{ex}[thm]{Example}

\theoremstyle{remark}

\newtheorem{rem}[thm]{Remark}

\numberwithin{equation}{section}

\def\Mat{{\mathcal Mat}}
\def\ST{{\mathcal ST}}
\def\cyc{ {\operatorname {cycle}}}
\def\fract{ {\operatorname {frac}}}
\def\cofrac{ {\operatorname {cofrac}}}
\def\vertAux{ {\operatorname {vertAux}}}
\def\PertAux{ {\operatorname {PertAux}}}
\def\td{ {\operatorname {td}}}

\subjclass[2010]{Primary 52B12; Secondary 52A20, 05A15}
\address{Department of Mathematics, University of California, One Shields Avenue, Davis, California 95616.}
\email{fuliu@math.ucdavis.edu}

\begin{document}
\title{Perturbation of transportation polytopes}
\author{Fu Liu}\thanks{The author is partially supported by the Hellman Fellowship from UC Davis.} 
\begin{abstract}
We describe a perturbation method that can be used to reduce the problem of finding the multivariate generating function (MGF) of a non-simple polytope to computing the MGF of simple polytopes. We then construct a perturbation that works for any transportation polytope. We apply this perturbation to the family of central transportation polytopes of order $kn \times n,$ and obtain formulas for the MGFs of the feasible cone of each vertex of the polytope and the MGF of the polytope. The formulas we obtain are enumerated by combinatorial objects. A special case of the formulas recovers the results on Birkhoff polytopes given by the author and De Loera and Yoshida. We also recover the formula for the number of maximum vertices of transportation polytopes of order $kn \times n.$
\end{abstract}

\maketitle

\section{Introduction}

Counting the number of lattice points inside a polytope is a problem of much interest. One method that is widely used to obtain the number of lattice points of a rational polytope $P$ is via its so-called ``multivariate generating function'' or ``MGF''. A benefit of this approach is that the method can be extended to obtain the volume of $P$ and the number of lattice points in the dilations of $P$ from its MGF as well. Therefore, the multivariate generating function is an interesting and useful object to study. 

The most common approach to computing the MGF of a polytope (or a polyhedron) $P$ is to compute the MGF of the feasible cone of $P$ at each vertex and then sum them up, %(see Theorem \ref{brion} and Corollary \ref{cor:intepoly}), 
where the feasible cone of $P$ at a vertex $v$ is the cone generated by the edge directions coming out of $v.$
However, computations involving non-simple cones are often complicated.
When a cone is non-simple, a common method used to compute its MGF involves triangulating the cone into simple cones \cite{BarviPom, barvinokIntPts}. Perturbation is another method often used to deal with non-simple cones in computational mathematics \cite{dantzig, galKruseZornig, kruse}, but it has not been widely used for computing MGFs. (However, see \cite{brion-vergne3, szenesVergne} for some work in which perturbation is used for computing MGFs.)
In the first part of our paper, we introduce a framework for perturbing polyhedra, with which we can compute the MGFs of non-simple polyhedra. 

Although there are general formulas for counting lattices points of polytopes, there is a lot of interest in finding more explicit formulas for special families of polytopes. In the second half of our paper, we focus on the family of ``transportation polytopes''. We construct an explicit perturbation that works for any transportation polytope and then apply our perturbation method for computing MGFs to obtain combinatorial formulas for MGFs of transportation polytopes.

We now discuss our results in more detail. 
We start with the following theorem describing a framework for globally perturbing polyhedra, with which we can compute the MGFs of non-simple polyhedra. Although the theorem is equivalent to techniques known to the experts in the field, the particular presentation is well suited for our purposes, and we are not aware of it being in the literature in this form.

\begin{thm}\label{thm:funcentire}
Let $A$ be an $N \times D$ matrix and $\b \in \R^N.$ 
Suppose $P$ is a non-empty integral polyhedron in $\R^D$ defined by $A \x \le \b$ and $\b(t)$ is a continuous function on some interval containing $0$ such that $\b(0)=\b.$ Suppose for each $t \neq 0$ in the interval, $\b(t)$ defines a non-empty polyhedron $P(t) = \{ \x \ | \ A \x \le \b(t)\}$ with exactly $\ell$ vertices: $w_{t,1}, \dots, w_{t, \ell},$ and the feasible cone $\fcone(P(t), w_{t,j})$ of $P(t)$ at $w_{t,j}$ does not depend on $t,$ that is, for each $j: 1 \le j \le \ell,$ there exists a fixed cone $K_j$ such that $\fcone(P(t), w_{t,j}) = K_j$ for all $t \neq 0.$ Then 
	\[ f(P, \z) =  \sum_{j=1}^\ell \z^{\lim_{t \to 0} w_{t,j}} f(K_j, \z).\]Here $f(P,\z)$ and $f(K_j,\z)$ denote the multivariate generating functions of $P$ and $K_j$, respectively. 
\end{thm}
See Definition \ref{defn:mgf} for the formal definition of multivariate generating function.
Note that the convergence of $w_{t,j}$ is proved in Theorem \ref{thm:func}, which also provides a local version for the perturbation method. According to Theorem \ref{thm:funcentire}, if we are given a non-simple integral polytope $P$ and are able to find a way of perturbing $P$ into simple polytopes satisfying conditions of Theorem \ref{thm:funcentire}, we can use the MGF of the feasible cones of the perturbed polytope $P(t)$ to calculate the MGF of $P.$ (See Example \ref{ex:thmfunc} for an example of how to use Theorems \ref{thm:funcentire} and \ref{thm:func}.) 
In \cite{birkhoff}, the authors find a combinatorial expression for the MGF of the Birkhoff polytope $B_n$, from which they obtain the first combinatorial formulas for the volume and Ehrhart polynomial of $B_n$ by residue calculations. In the present paper, we apply our perturbation method to generalize the results in \cite{birkhoff}. 
Before stating our results, we recall some definitions.

%Let $B_n$ be the convex polytope of $n \times n$ doubly-stochastic matrices; that is the set of real nonnegative matrices with all row and column sums equal to one. The volume and Ehrhart polynomial of $B_n$ are extremely hard to compute. The volume $\vol(B_n)$ has been computed for $n \le 10$ \cite{beckpixton} and the Ehrhart polynomial of $B_n$ has only been computed for $n \le 9.$ In \cite{birkhoff}, the authors find a combinatorial expression for the multivariate generating function (MGF) of $B_n$, from which they obtain the first combinatorial formulas for the volume and Ehrhart polynomial of $B_n.$ The majority work in \cite{birkhoff} is to find the MGF of $B_n,$ from which the authors obtain formulas for the volume and Ehrhart polynomial of $B_n$ by residue calculation.

The Birkhoff polytope belongs to the family of {\it transportation polytopes}:
Given $\br = (r_1, \dots, r_m)$ and $\bc = (c_1, \dots, c_n)$ two vectors of positive entries whose coordinates sum to a fixed number, the transportation polytope determined by $\br$ and $\bc$, denoted by $\cT(\br, \bc)$, is the set of all $m \times n$ nonnegative matrices in which row $i$ has sum $r_i$ and column $j$ has sum $c_j$. We call $\cT(\br,\bc)$ a transportation polytope of {\it order $m \times n.$} The problem of finding MGFs of transportation polytopes is relatively easy when a transportation polytope is ``non-degenerate'' (see Section \ref{sec:proptrans} for the definition), in which case one can apply Corollary \ref{cor:mgfnondeg} to find its MGF. 
However, if a transportation polytope is degenerate, it is usually non-simple, in which case one often has to triangulate the non-simple feasible cones. We describe a perturbation (in Lemma \ref{lem:unipert}) for all the transportation polytopes so that the problem is reduced to finding the MGF of non-degenerate cases. In addition to this, our perturbation provides families of transportation polytopes that obtain the maximum possible number of vertices among all the transportation polytopes of the same order (see Lemma \ref{lem:maxvert}).

%the common method used involves triangulating the feasible cone of each vertex, which is more complicated. One of the main results of this paper is a perturbation method that can be used to find the MGFs of degenerate transportation polytopes.

%As the simplest generalization of Birkhoff polytopes, it is natural to study the same questions of volume and Ehrhart polynomial for the family of transportation polytopes.
%Since the techniques used in \cite{birkhoff} to obtain volumes and Ehrhart polynomials from MGFs can be applied to the MGF of any polytope, we will focus on finding MGFs of transportation polytopes. 

Addressing a question of Bernd Sturmfels, we consider a special family of transportation polytopes that contains the Birkhoff polytope: the family of {\it central transportation polytopes}. A classical central transportation polytope of order $m \times n$ is a transportation polytope whose column sums are all $m$ and row sums are all $n.$ However, strictly speaking Birkhoff polytopes do not belong to this family because the column sums and row sums are all $1$ for $B_n.$ Therefore, we slightly generalize the definition of the classical central transportation polytopes to include Birkhoff polytopes. A transportation polytope $\cT(\br,\bc)$ is {\it central} of order $m \times n$ if all the column sums are the same and all the row sums are the same. The family of central transportation polytopes is an interesting subset of transportation polytopes. For example, when $m$ and $n$ are coprime, the central transportation polytope of order $m \times n$ achieves the maximum possible number of vertices among all the transportation polytopes of order $m \times n$ \cite{bolker}. In \cite{birkhoff}, the combinatorial data used to enumerate the MGF of $B_n$ is the family of rooted trees. %Although more than half of the central transportation polytopes are totally unimodular polytopes (see Defintion \ref{defn:totuni}), whose MGFs are relatively easy to find (using Cor \ref{cor:totuni}), it is not clear that all the MGFs can be enumerated by combinatorial data. 
Sturmfels asked whether we can give a nice description in terms of trees for the MGF of any central transportation polytope. We answer his question for central transportation polytopes of order $kn \times n$ by using the perturbation method we develop. 

Denote by $\ST_{k,n}$ the set of spanning trees of the complete bipartite graph $K_{kn,n}$ with right degree sequence $(k+1, k+1, \dots, k+1, k).$ (See Definition \ref{defn:degree} for the definition of right degree sequence.)
\begin{thm}\label{thm:mgfcase2} Assume that $\cT(\br,\bc)$ is a central transportation polytope of order $m \times n$, where $m = kn,$ and $\br = (a, \dots, a)$ and $\bc = (b, \dots, b)$ are two integer vectors. 

%Suppose $M \in \Mat_{k,n}.$ 
Suppose $M$ is a vertex of $\cT(\br,\bc).$ Then the MGF of the feasible cone $\fcone(\cT(\br,\bc), M)$ of $\cT(\br,\bc)$ at $M$ is
\begin{equation}\label{equ:mgfcase20}
	f(\fcone(\cT(\br,\bc), M), \z) =\sum_{T \in \PertAux(M)} \prod_{e \not\in E(T)} \frac{1}{1 - \z^{\cyc(T,e)}}.
\end{equation}
Thus, the MGF of $\cT(\br,\bc)$ is
%\[f(\cT(\br, \bc),\z) =\sum_{M: \text{a vertex of } \cT(\br,\bc) } \z^{M} \sum_{T \in \PertAux(M)} \prod_{e \not\in E(T)} \frac{1}{1 - \z^{\cyc(T,e)}}. \]
\begin{equation}\label{equ:mgfcase2entire1}
	f(\cT(\br, \bc),\z) =\sum_{T \in \ST_{k,n}} \z^{M_T} \prod_{e \not\in E(T)} \frac{1}{1 - \z^{\cyc(T,e)}}. 
\end{equation}
	In both equations, $\cyc(T,e)$ denotes the $m \times n$ $(0,-1,1)$-matrix associated to the unique cycle in the graph $T \cup e$ (see Definition \ref{defn:cyc} for details). In the first equation,
%\begin{eqnarray*}
%	\PertAux(M) &=& \{ \text{ all the spanning trees of complete bipartite graph $K_{kn,n}$ with} \label{equ:pertauxdes} \\
%	& & \quad \text{right degree sequence $(k+1, \dots, k+1, k)$ that contains $\aux(M)$ }\}.
%\end{eqnarray*}
\begin{equation*}
	\PertAux(M) = \{ T \in \ST_{k,n} \ | \ \text{ the auxiliary graph of $M$ is a subgraph of $T$},\}
\end{equation*}
and in the second equation, $M_T$ is the unique vertex of $\cT(\br,\bc)$ whose auxiliary graph is a subgraph of $T.$
See Definition \ref{defn:aux} for the definition of auxiliary graph.
\end{thm}

Note that the Birkhoff polytope $B_n$ is a special case of this family of central transportation polytopes when $k=1.$ Our result recovers the formulas for the MGF of $B_n$ given in \cite{birkhoff}.

We remark that in \cite{sturmfelsWeisZie}, the authors also provided a framework for global perturbation method using the idea of lattice ideals (see \cite[Corollary 7.5]{sturmfelsWeisZie}). The fundamentals of these two theories are similar; we just use different approaches. However, in our paper, we also provide a local version for the perturbation method in Theorem \ref{thm:func}. %One question might be interesting to ask is whether the techniques used in \cite{sturmfelsWeisZie} also lead to a similar result.

This paper is organized as follows. In Section \ref{sec:background}, we give background on results related to multivariate generating functions. In Section \ref{sec:perturb}, we give the framework for our perturbation method. In Section \ref{sec:proptrans}, we review results on properties of transportation polytopes. In Section \ref{sec:unipert}, we give a perturbation that works for every transportation polytope. In Section \ref{sec:central}, we apply the perturbation method to central transportation polytopes of order $kn \times n$ and give combinatorial formulas for the MGFs, Ehrhart polynomials, and volumes of these polytopes. When we specialize these formulas to Birkhoff polytopes, we recover the results in \cite{birkhoff}. We also recover the formula for the maximum possible number of vertices of transportation polytopes of order $kn \times n.$

\section{Background}\label{sec:background}
A {\it polyhedron} is the set of points defined by a system of linear inequalities 
\[ \n_i(\x) \le b_i, \ \forall 1 \le i \le N,\]
where $\n_i$ is a $D$-vector and $\n_i( \cdot)$ is the linear function that maps $\x \in \R^D$ to the dot product of $\n_i$ and $\x.$ Since we can consider $\n_i( \cdot)$ as a point in the dual space $(\R^n)^*$ of $\R^n,$ we will use the notation $\n_i$ (or any bold letter) to denote both the $D$-vector and the linear function.
For simplicity, we let $A$ be the $N \times D$ matrix whose row vectors are $\n_i$'s and $\b = (b_1,\dots, b_N)^T,$ so the system of linear inequalities can be represented as \[ A \x \le \b.\]
%Through out this paper, when we talk about general polyhedron, we always assume it is defined by $A \x \le \b,$ where the row vectors of $A$ are $\n_1, \dots, \n_N.$

A {\it polytope} is a bounded polyhedron. We assume familiarity with basic definitions of polyhedra and polytopes as presented in \cite{grunbaum, zie}. 
For any polyhedron $P,$ we use $\vert(P)$ to denote the vertex set of $P.$
An {\it integral} polyhedron is a polyhedron whose vertices are all lattice points, i.e., points with integer coordinates. A {\it rational} polyhedron is a polyhedron whose vertices are all rational points.

\subsection{Ehrhart polynomials and multivariate generating functions}

For any polytope $P \subset \R^D,$ and a nonnegative integer $t,$ we define 
\[ i(P,t) = \# (tP \cap \Z^D)\] to be the number of lattice points inside $t P = \{ t \x \ | \ \x \in P\},$ the {\it $t$th dilation} of $P.$ It is well-known that given a $d$-dimensional integral polytope $P,$ the function $i(P,t)$ is a polynomial in $t$ of degree $d$ with leading coefficient being the normalized volume of $P.$ Since this was first discovered by Ehrhart \cite{ehrhart}, we often refer to $i(P,t)$ as the {\it Ehrhart polynomial} of $P.$
Because the leading coefficient of $i(P,t)$ gives the volume of $P,$ obtaining the Ehrhart polynomials of polytopes is one way people use to compute volumes of polytopes.
One can find the Ehrhart polynomial $i(P,t)$ of $P$ using the multivariate generating function.

\begin{defn}\label{defn:mgf}
	Let $P \subset \R^D$ be a polyhedron. The {\it multivariate generating function} (or {\it MGF}) of $P$ is:
\[ f(P, \z) = \sum_{\alpha \in P \cap \Z^D} \z^\alpha, \]
where $\z^\alpha = \prod_{i=1}^D z_i^{\alpha_i}.$
\end{defn}
%\[ f(mP, \z) = \sum_{\x \in mP \cap \Z^D} \z^\alpha,\]
%where $\z^\alpha = \prod_{i=1}^N z_i^{\alpha_i}.$ 
\begin{ex}\label{ex:triangle}
Suppose $P$ is the triangle defined by $x \ge 0, y \ge 0$ and $x+y \le 2.$ It contains $6$ lattice points: $(0,0), (1,0), (2,0), (0,1), (1, 1)$ and $(0,2).$ Hence, its MGF is 
\[ f(P, \z) = z_1^0 z_2^0 + z_1^1 z_2^0 + z_1^2 z_2^0 + z_1^0 z_2^1 + z_1^1 z_2^1 + z_1^0 z_2^2 = 1 + z_1 + z_1^2 + z_2 + z_1 z_2 + z_2^2.\]
\end{ex}

Note that when $P$ is a polytope, we obtain $i(P, t)$ by plugging $z_i =1$ for all $i$ in the MGF $f(tP, \z)$ of the $t$th dilation $t P$ of $P.$

\subsection{Cones of polyhedra}
One benefit of computing the MGF of a polyhedron is that the problem can be reduced to computing the MGF of the tangent/feasible cones of the given polyhedron by applying Brion's theorem. Let's first review some related definitions and results.

\begin{defn}
%For each $i,$ let $H_i := \{ \x \ | \ \n_i(\x) = b_i\}.$ We call $H_i$'s the {\it defining hyperplanes} for $P.$
%For any face $f$ of $P,$ if $f \subseteq H_i$ for some $i,$ we say $H_i$ is a {\it supporting hyperplane} of $f.$
Suppose $P$ is a polyhedron defined by $A \x \le \b,$ where the row vectors of $A$ are $\n_1, \dots, \n_N.$
Let $F$ be a face of $P.$ If $\n_i(\x) = b_i$ for any $\x \in F,$ we say the inequality $\n_i(\x) \le b_i$ {\it supports} the face $F$ or $\n_i(\x) = b_i$ is a {\it supporting hyperplane} for $F.$

We denote by
\[ \cN(A, \b, F) = \{ \n_i \ | \ \n_i(\x)=b_i, \forall \x \in F \}\]
the set of the normal vectors $\n_i$ of all the inequalities among $A \x \le \b$ that support the face $F.$ %If we consider $\n_i$'s are linear functions, the set $\cN(A, \b, f)$ is the set of linear functions $\n_i$ whose maximum on $P$ is attained on $f.$
%the set of linear functions $\n_i$ that $f$ achieves the maximum value among all the points in $P,$ or equivalently, the set of row vectors $\n_i$ of $A$ satisfying $\n_i(\x) \le b_i$ achieves equality for each $\x \in f.$ 
\end{defn}
%Clearly, if $i \not\in \cN(A, \b, f),$ we have that $\n_i(\x) < b_i$ for some $\x \in f.$

\begin{defn}
	Let $K \subset \R^D$ be a cone. The {\it polar cone} of $K$ is the cone 
	\[ K^\circ = \{ \y \in (\R^D)^* \ | \ \y(\x) \le 0,  \forall \x \in K\}.\]
\end{defn}

\begin{defn}\label{defn:vertcone}
	Suppose $P$ is a polyhedron and $v \in P.$ The {\it tangent cone} of $P$ at $v$ is $$\tcone(P,v) = \{v+u \ | \ v+ \delta u \in P \mbox{ for all sufficiently small $\delta >0$}  \}.$$
	The {\it feasible cone} of $P$ at $v$ is $$\fcone(P,v) = \{u \ | \ v+ \delta u \in P \mbox{ for all sufficiently small $\delta >0$}  \}.$$
\end{defn}

Let $v$ be a vertex of $P.$ One can check that the polar cone of the feasible cone of $P$ at $v$ is the cone generated by the vectors in $\cN(A, \b, v):$ 
\begin{equation}\label{equ:gendfcone}
	(\fcone(P,v))^\circ = \Cone(\cN(A, \b, v)).
\end{equation}

\begin{defn}
	Suppose $P$ is a non-empty polyhedron defined by $A \x \le \b.$ The {\it recession cone} of $P,$ denoted by $K_P$, is
	\[ K_P = \{ \x \ | \ A \x \le 0 \}.\]
\end{defn}

\subsection{Indicator functions and Brion's theorem}

For a set $S \subseteq \R^D$, the indicator function $[ S ]: \R^D
\rightarrow \R$ of $S$ is defined as $$[ S ] (x) = \left \{
\begin{array}{ll} 1 \mbox{ if }x \in S, \\ 0 \mbox{ if }x \not \in
S.\\
\end{array}\right .$$ 

We assume the readers are familiar with the definition of algebra of polyhedra/polytopes and valuation presented in \cite{BarviPom}. The following lemma gives the two important equations of indicator functions of cones of polyhedra. 
\begin{lem}[Theorem 6.4 and Problem 6.2 in \cite{barvinokIntPts}] Suppose $P$ is a non-empty polyhedron without lines. Then
\begin{align} 
	[ P ] \equiv& \sum_{v \in \vert(P)} [\tcone(P, v)] \ \text{ modulo polyhedra with lines;} \label{equ:tcone}\\
	{[}  K_P ] \equiv& \sum_{v \in \vert(P)} [\fcone(P, v)] \ \text{ modulo polyhedra with lines,} \label{equ:fcone}
\end{align}
	where $K_P$ is the recession cone of $P.$

\end{lem}

The following lemma gives a connection between indicator functions of cones and indicator functions of their polars. 
\begin{lem}[Corollary 6.5 in \cite{barvinokIntPts}] 
\label{lem:equvipolar}
%Let $P_i, i \in I$ be a family of polyhedra such that $0 \in P_i$ for all $i \in I$ and let $\alpha_i$ be numbers. Then
Let $\{K_i \ | \ i \in I\}$ be a family of cones and let $\alpha_i$ be numbers. Then
%\[ \sum_{i \in I} \alpha_i [P_i] \equiv 0 \ \text{ modulo polyhedra with lines} \]
\[ \sum_{i \in I} \alpha_i [K_i] \equiv 0 \ \text{ modulo polyhedra with lines} \]
if and only if
%\[ \sum_{i \in I} \alpha_i [P_i^\circ] \equiv 0 \ \text{ modulo polyhedra in proper subspaces.} \]
\[ \sum_{i \in I} \alpha_i [K_i^\circ] \equiv 0 \ \text{ modulo polyhedra in proper subspaces.} \]
\end{lem}	

Apply the above lemma to Equation \eqref{equ:fcone} gives us the following, which will be useful to us later:	
\begin{cor}
Suppose $P$ is a non-empty polyhedron without lines. 
\begin{equation}
	{[}  K_P^\circ ] \equiv \sum_{v \in \vert(P)} [(\fcone(P, v))^\circ] \ \text{ modulo polyhedra in proper subspaces.} \label{equ:dfcone}
\end{equation}
\end{cor}

It turns out that the multivariate generating functions define a valuation on the algebra of polyhedra. %(See \cite{BarviPom} for the definitions of polyhedra algebra, indicator function, and valuation.)

\begin{thm}[Theorem 3.1 and its proof in \cite{BarviPom}] \label{thm:rational}
There is a map $\mathfrak F$ which, to each rational polyhedron $P
\subset \R^d$, associates a unique rational function $f(P,\z)$ in $D$ complex
variables $\z \in \C^D,$ $\z = (z_1, \dots, z_D),$ such that the
following properties are satisfied:
\begin{itemize}
\item[(i)] The map $\mathfrak F$ is a valuation. %(Please see
%  \cite{BarviPom} for details about valuation.)
\item[(ii)] If $P$ is pointed, there exists a nonempty open subset
  $U_p \subset \C^D,$ such that $\sum_{\alpha \in P \cap \Z^D} \z^{\alpha}$ converges absolutely to
  $f(P, \z)$ for all $\z \in U_P.$ 
\item[(iii)] If $P$ is pointed,
then $f(P,\z)$ satisfies
$$f(P,\z) = \sum_{\alpha \in P \cap \Z^D} \z^{\alpha}$$ for any $\z
\in \C^D$ where the series converges absolutely.
\item[(iv)] If $P$ is not pointed, i.e., $P$ contains a line, then $f(P, \z) = 0.$ 
\end{itemize}
\end{thm}

Using this valuation property and Equation \eqref{equ:tcone}, we immediately have Brion's theorem:

\begin{thm}[Brion, 1988;  Lawrence, 1991]
	%(see \cite{BarviPom,beckhassesottile} for proofs) 
\label{brion} Let $P$ be a rational polyhedron. Then, considered as rational functions,
$$f(P,\z) = \sum_{v \in \vert(P)} f(\tcone(P,v),\z).$$
\end{thm}

\begin{cor}\label{cor:intepoly} If $P$ an {integral polyhedron}, %i.e., all the vertices of $P$ are lattice points, 
then
\begin{equation}\label{equ:mgfintepoly}
	f(P,\z) = \sum_{v \in \vert(P)} \z^v f(\fcone(P,v),\z).
\end{equation}
Hence, for any positive integer $t,$
\begin{equation}\label{equ:mgfintedialpoly}
f(t P,\z) = \sum_{v \in \vert(P)} \z^{t v} f(\fcone(P,v),\z).
\end{equation}
\end{cor}

\begin{ex}
Let $P$ be the integral interval $[2, 5],$ which is a $1$-dimensional polytope whose vertices are $v_1 = 2$ and $v_2 = 5.$

The feasible cone of $P$ at $v_1$ is just the nonnegative $x$-axis: $\fcone(P, v_1) = \{ x : x \ge 0\}.$ Thus, the lattice points in $\fcone(P,v_1)$ are all the nonnegative integers. Hence,
\[ f(\fcone(P,v_1), z) = z^0 + z^1 + z^2 + \cdots  \frac{1}{1-z},\]
where the second equality holds for $|z| < 1.$
Similarly, we can get
\[ f(\fcone(P,v_2), z) = z^0 + z^{-1} + z^{-2} + \cdots = \frac{1}{1-z^{-1}} =\frac{z}{z-1},\]
where the second equality holds for $|z| > 1.$
Brion theorem states that considered as rational functions, we have
\begin{align*}
	f(P, z) =& z^{v_1} f(\fcone(P,v_1), z) + z^{v_2} f(\fcone(P,v_1), z) \\
	=& z^2 \ \frac{1}{1-z} + z^5 \ \frac{z}{z-1} = \frac{z^2 - z^6}{1-z} = z^2 + z^3 + z^4+z^5,
\end{align*}
which agrees with what one would get by computing $f(P,z)$ directly using Definition \ref{defn:mgf}.
\end{ex}

By Theorem \ref{brion}, the problem of finding the MGF of a rational polyhedron $P$ is reduced to finding the MGF of each of its tangent cone. When $P$ is integral, the problem is reduced further: by Corollary \ref{cor:intepoly}, it suffices to find the formulas for the MGF of the feasible cone of each vertex of $P.$

\subsection{MGFs of unimodular cones}
In general, one cannot calculate the MGF of a cone just by reading its generating rays. If a cone is not {\it simple}, i.e., the number of rays that generate the cone is larger than the dimension of the cone, one usually has to triangulate the cone into simple cones first. Even if a cone is simple, it is usually impossible to calculate its MGF directly from its generating rays. However, it can be done when the cone is {\it unimodular}.  
A pointed cone $K$ in $\R^D$ generated by the rays $\{r_i\}_{1 \le i \le d}$ is {\it unimodular} if $r_i$'s form a $\Z$-basis of the lattice $\Z^D \cap \Span(K)$.
\begin{lem}[Lemma 4.1 in \cite{BarviPom}] \label{lem:unicone}
Suppose $K$ is a unimodular cone generated by the rays $\{r_i\}_{1 \le i \le d}$. Then
$$f(K,\z) = \prod_{i=1}^d \frac{1}{1 - \z^{r_i}}.$$
\end{lem}

%Because computing the MGF of a unimodular cone is easy, the common method of computing the MGF of a cone is to express the indicator function of the cone as signed sum of indicator functions of a set of unimodular cones. 
Because computing the MGF of a unimodular cone is easy, it is easy to compute the MGF of an integral polyhedron/polytope whose feasible cones are all unimodular. Therefore, we give the following definition.

\begin{defn}\label{defn:totuni}
	A polytope $P \subset \R^D$ is {\it totally unimodular} if every vertex of $P$ is a lattice point and every feasible cone of $P$ is unimodular.
\end{defn}

\begin{cor}\label{cor:totuni}
Suppose $P \subset \R^D$ is a totally unimodular polytope. Then 
\[ f(P, \z) = \sum_{v \in \vert(P)} \z^v \prod_{i=1}^d \frac{1}{1- \z^{r_{v,i}}},\]
where $r_{v,1}, \dots, r_{v,d}$ are the generating rays of the vertex $v.$
\end{cor}

\begin{proof}
	It follows from Corollary \ref{cor:intepoly} and Lemma \ref{lem:unicone}.
\end{proof}

\subsection{Obtaining volume and Ehrhart polynomial from MGF} Barvinok gives a polynomial algorithm to express a simple cone as signed sum of unimodular cones \cite[Theorem 4.2]{BarviPom}. 
Applying this algorithm to each simple cone in a triangulation of a cone, one can get the MGF of the cone. Using this idea and Theorem \ref{brion}, Barvinok \cite[Theorem 4.4]{BarviPom} shows that for any $P \subset \R^D$ a $d$-dimensional rational polyhedron, the multivariate generating function of $P$ is in the form of
\begin{equation}\label{equ:mgf}
f(P, \z) = \sum_{i} \epsilon_i
\frac{\z^{v_i}}{(1-\z^{r_{i,1}})\cdots(1-\z^{r_{i,d}})},
\end{equation}
 where $\epsilon_i = \{-1, 1\},$ $v_i,r_{i,1}, \dots, r_{i,d} \in \Z^D$, 
the $v_i$'s are all vertices (with multiple occurrences) of $P,$
and $\Cone(r_{i,1}, \dots, r_{i,d})$ is unimodular, for each $i.$

Because the above formula is obtained via Theorem \ref{brion}, by Corollary \ref{cor:intepoly}, if $P$ is an integral polytope, we only need to replace $v_i$ with $t v_i$ in \eqref{equ:mgf} to get the MGF of its dilation $tP$. Then one can use the residue calculation showed in \cite{BarviPom} to find the Ehrhart polynoial and the volume of $P.$

\begin{lem}[Lemma 5.4 in \cite{birkhoff}]\label{lem:mgf2volehr}
	Suppose $P \subset \R^D$ a $d$-dimensional integral polytope and the MGF of $P$ is in the form of \eqref{equ:mgf}. Then for any choice of $c \in \R^D$ such that $\langle c, r_{i,j} \rangle \neq 0$ for each $i$ and $j,$ the Ehrhart polynomial of $P$ is
\begin{equation}\label{equ:ehrhart}
i(P,t) = \sum_{k=0}^d \frac{t^k}{k!}\sum_i \frac{\epsilon_i}{\prod_{j=1}^d \langle c , r_{i,j} \rangle} {(\langle c, v_i \rangle)^k} \td_{d-k}(\langle c , r_{i,1} \rangle, \dots, \langle c , r_{i,d} \rangle),
\end{equation}
where $\td_{d-k}( \ )$ is a Todd polynomial. (See \cite{BarviPom} or \cite{birkhoff} for definition of Todd polynomials.)

In particular, the normalized volume of $P$ is
\begin{equation}\label{equ:vol}
\vol(P) = \frac{1}{d!}\sum_i \epsilon_i \frac{(\langle c, v_i \rangle)^d}{\prod_{j=1}^d \langle c , r_{i,j} \rangle} .
\end{equation}
\end{lem}
Hence, the problem of finding formulas for the volume and Ehrhart polynomial of an integral polytope is reduced to finding the formula for its MGF. Therefore, our paper will be mainly focused on computing the MGFs of polytopes.

\section{A perturbation method}\label{sec:perturb}
When calculating the MGF of a polytope/polyhedron which has non-simple feasible cones, we usually triangulate those non-simple feasible cones into simple cones, and then apply various algorithms \cite[Chapter 16]{barvinokIntPts} for computing MGFs of simple cones to find the final formula. 
In this section, we describe a perturbation method that can be used to replace the triangulation step in the above procedure.
The method is particularly useful when the perturbed polytopes are totally unimodular, in which case instead of using other algorithms, it suffices to use Lemma \ref{lem:unicone} or Corollary \ref{cor:totuni} to obtain the MGFs.

%We then apply this method to transportation polytopes. Because the constraint matrix $A_{m,n}$ of transportation polytopes is totally unimodular, the simple cones we obtain from the perturbation are all unimodular. Hence, instead of using other algorithms, it suffices to use Corollary \ref{lem:unicone} to obtain the MGFs.

\begin{thm}\label{thm:func}
Let $A$ be an $N \times D$ matrix and $\b \in \R^N.$ 
Suppose $P$ is a non-empty polyhedron in $\R^D$ defined by $A \x \le \b$ and $\b(t)$ is a continuous function on some interval containing $0$ such that $\b(0)=\b.$ Suppose for each $t \neq 0$ in the interval, $\b(t)$ defines a non-empty polyhedron $P(t) = \{ \x \ | \ A \x \le \b(t)\}$ with exactly $\ell$ vertices: $w_{t,1}, \dots, w_{t, \ell},$ and the feasible cone of $P(t)$ at $w_{t,j}$ does not depend on $t,$ that is, for each $j: 1 \le j \le \ell,$ there exists a fixed cone $K_j$ such that $\fcone(P(t), w_{t,j}) = K_j$ for all $t \neq 0.$ Then we have the following:
\begin{ilist}
	\itm For each fixed $j: 1 \le j \le \ell,$ the set of vertices $\{ w_{t,j} \}$ converges to some vertex of $P$ as $t$ goes to $0.$
	\itm	For each vertex $v \in \vert(P)$, let $J_v$ be the set of $j$'s where $\{ w_{t,j} \}$ converges to $v.$ Then
	\[ [ \fcone(P, v) ] \equiv \sum_{j \in J_v} [ K_j ] \text{ modulo polyhedra with lines.} \]
	Therefore,
	\[ f(\fcone(P, v), \z) = \sum_{j \in J_v} f(K_j, \z). \] 

\end{ilist}
\end{thm}

One sees that Theorem \ref{thm:funcentire} follows directly from Theorem \ref{thm:func} and Corollary \ref{cor:intepoly}. 

The phenomenon described in Theorem \ref{thm:func} is similar to the one in Theorem 18.3 of \cite{barvinokIntPts}. However, the conditions given in these two theorems are different. (For example, the convergence of the vertices $w_{t,j}$ is a hypothesis in Theorem 18.3 of \cite{barvinokIntPts} but is a conclusion in our theorem.) One might be able to modify the proof of Theorem 18.3 of \cite{barvinokIntPts} to give a proof for part (ii) of Theorem \ref{thm:func} provided that (i) is proved. 
We also remark that one can deduce part (ii) of Theorem \ref{thm:func} from the proposition (on page 818) in Section 3.1 of \cite{brion-vergne3}. However, since the presentation we need is slightly different, we find it is easier to give an independent self-contained proof of the theorem in this article. 
%However, we give an independent proof in our paper. 
Before we prove Theorem \ref{thm:func}, we give an example demonstrating how this theorem (and Theorem \ref{thm:funcentire}) work and then give a preliminary lemma.

\begin{ex}[Example of Theorems and \ref{thm:funcentire} and \ref{thm:func}]  \label{ex:thmfunc}
Let 
\[ A = 
\begin{pmatrix} 
	1 & 0 \\
	-1 & 0 \\
	0 & 1 \\
	0 & -1 \\
	1 & -1 \\
	1 & 2 \\
	2 & 1 
\end{pmatrix}, \ 
\b = \begin{pmatrix}1 \\ 0 \\ 1 \\ 0 \\ 1 \\ 3 \\ 3 \end{pmatrix}, \
	\text{ and }
\b(t) = \begin{pmatrix}1 \\ t \\ 1 \\ 0 \\ 1 - 2t \\ 3 - 3t \\ 3 - 3t \end{pmatrix}, 0 \le t < 1/5. 
\]
Then $P = \{ \x \in \R^2 \ | \ A \x \le \b\}$ is just the unit square with vertices 
\[ v_1=(0,0), v_2=(0,1), v_3=(1,1) \text{ and } v_4=(1,0).\]
For any $t \in (0, 1/5),$ the polygon $P(t)$ has seven vertices 
\[	w_{t,1} = (t,0), \quad w_{t,2}=(t,1),\]
\[w_{t,3}=(1-3t, 1), \quad w_{t,4}=(1-t,1-t), \quad w_{t,5}=(1,1-3t),\]
\[w_{t,6}= (1, 2t), \quad w_{t,7}=(1-2t,0).\]
See Figure \ref{fig:exconvcube}. It shows polygons $P(3/20), P(1/10)$ and $P(1/20).$ One can see that as $t$ decreases, the polygons look more and more similar to the unit cube $P=P(0).$ In this example, as $t$ goes to $0,$ 
\[ w_{t,1} \to v_1, \quad w_{t,2} \to v_2, \quad w_{t,3},w_{t,4},w_{t,5} \to v_3, \quad w_{t,6}, w_{t,7} \to v_4.\]
Hence, \[ J_{v_1} = \{1\}, \quad J_{v_2} =\{2\}, \quad J_{v_3} =\{3,4,5\}, \quad J_{v_4} =\{6,7\}.\]
Note that the feasible cone of $P(t)$ at $w_{t,j}$ does not depend on $t$ (for $t \in (0, 1/5)$). Let $K_j := \fcone(P(t), w_{t,j}).$ By Theorem \ref{thm:func}, we have
\begin{align}
	{[} \fcone(P, v_1)] \equiv& \ [K_1], \quad \text{ modulo polyhedra with lines;} \label{equ:ex1} \\
	{[} \fcone(P, v_2)] \equiv& \ [K_2], \quad \text{ modulo polyhedra with lines;} \label{equ:ex2} \\
	{[} \fcone(P, v_3)] \equiv& \ [K_3]+[K_4]+[K_5], \quad \text{ modulo polyhedra with lines;} \label{equ:ex3} \\
	{[} \fcone(P, v_4)] \equiv& \ [K_6]+[K_7], \quad \text{ modulo polyhedra with lines.} \label{equ:ex4} 
\end{align}
We remark that Equations \eqref{equ:ex1} and \eqref{equ:ex2} are trivial, and actually hold with direct equalities (without modulo equivalence). Equations \eqref{equ:ex3} and \eqref{equ:ex4} are less trivial. 

Finally, we have the conclusion of Theorem \ref{thm:funcentire}:
\begin{align*}
	f(P,\z) =& \z^{v_1} f(K_1,\z) + \z^{v_2} f(K_2, \z) + \z^{v_3} \left( f(K_3,\z) + f(K_4,\z) + f(K_5), \z) \right) \\
	& + \z^{v_4} \left( f(K_6,\z) + f(K_7, \z) \right).
\end{align*}
\end{ex}

\begin{figure}
\centering
\scalebox{0.6}{\input{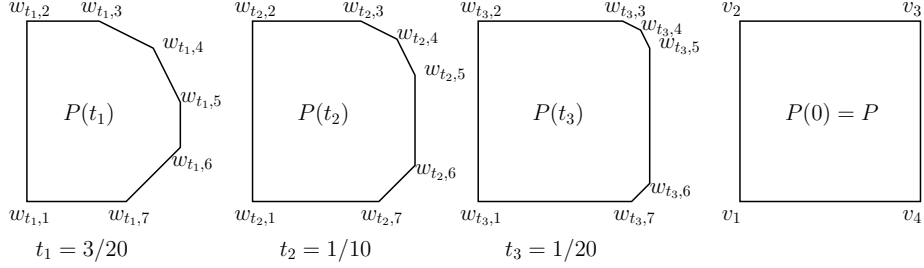}}
\caption{Example of polygons converging to the unit square}
\label{fig:exconvcube}
\end{figure}

We need the following lemma to prove Theorem \ref{thm:func}.
\begin{lem}\label{lem:vertconv}
Let $A$ be an $N \times D$ matrix with row vectors $\n_1, \dots, \n_N$ and $\b \in \R^N.$ 
Suppose $P$ is a non-empty polyhedron in $\R^D$ defined by $A \x \le \b$ and $\b(t)$ is a continuous function on some interval containing $0$ such that $\b(0)=\b.$ %$(c,0)$ or $(0,d)$ which converges to $\b$ as $t$ goes to $0.$ 
Let $P(t)$ be the polyhedron defined by the linear inequalities $A \x \le \b(t).$ Then we have the following:

\begin{ilist}
	\itm The vertex sets $\{ \vert(P(t)) \ | \ t \neq 0 \}$ {\textbf converge} to the vertex set $\vert(P)$ of $P;$ that is, for any $\epsilon > 0,$ there exists $\delta_\epsilon,$ such that for any $t: 0 < |t| < \delta_\epsilon,$ any vertex $w$ of $P(t),$ there exists a vertex $v$ of $P$ satisfying $|w - v| < \epsilon.$ 
	
	\itm Let $\epsilon$ be strictly smaller than one half of the minimum distances between two vertices of $P,$ and $t: 0 < |t| < \delta_\epsilon,$ where $\delta_\epsilon$ is defined as in {\rm (i)}. For any vertex $w$ of $P(t)$, the vertex $v$ of $P$ satisfying $|w-v|< \epsilon$ is unique, and such a pair of $w$ and $v$ satisfies
\begin{equation}\label{eq:Nsubset}
\cN(A, \b(t), w) \subseteq \cN(A, \b, v). 
\end{equation}
Hence, we have
\begin{equation}\label{eq:dfcsubset}
(\fcone(P(t), w))^\circ \subseteq (\fcone(P, v))^\circ.
\end{equation}
%and \begin{equation}\label{eq:fcsubset}
%\fcone(P(t), w) \supseteq \fcone(P, v).
%\end{equation}

\itm Choose $\epsilon$ and $t$ as in {\rm (ii)}. Suppose $P(t)$ is non-empty and has vertices. For any vertex $v \in \vert(P),$ let \[ W_{t,v} = \{ w \in \vert(P(t)) \ | \ |w-v| < \epsilon\}\] be the set of vertices of $P(t)$ that are within $\epsilon$-distance to $v.$
 Then
 \[ [ \fcone(P, v) ] \equiv \sum_{w \in W_{t,v}} [ \fcone(P(t), w) ] \text{ modulo polyhedra with lines.} \]
	Therefore,
	\[ f(\fcone(P, v), \z) = \sum_{w \in W_{t,v}} f(\fcone(P(t),w), \z). \] 

\end{ilist}

\end{lem}

\begin{proof}
Any $D$ linearly independent rows of $A$ form an invertible matrix. Let $M_1, \dots, M_k$ be all the invertible matrices arising this way. For each $j: 1 \le j \le k,$ let 
\[ I_j = \{\text{ the indices of the rows from $A$ that form matrix $M_j$} \},\]
and let $\b^j(t) = (b_i(t))_{i \in I_j}$ be the corresponding sub-vector of $\b(t)$ and let $\b^j = (b_i)_{i \in I_j}$ be the corresponding sub-vector of $\b.$ Note that $\b^j(t)$ is a continuous function where $\b^j(0) = \b^j.$
%that converges to $\b^j$ as $t$ goes to $0.$

Let $\epsilon > 0.$ For any $j: 1 \le j \le k,$ because $M_j^{-1}(\b^j(t))$ is a continuous function on $t,$
%converges to $M_j^{-1}(\b)$, 
there exists $\delta_j$ such that 
%for any $0 < t < \delta_j,$ we have $|M_j^{-1}(\b^j) - M_j^{-1}(\b^j(t))| < \epsilon.$
\begin{equation}\label{eq:deltaj}
0 < |t| < \delta_j \quad \Rightarrow \quad |M_j^{-1}(\b^j) - M_j^{-1}(\b^j(t))| < \epsilon.
\end{equation}

If $M_j^{-1}(\b^j)$ is not a vertex of $P,$ there exists $i \not\in I_j$ such that $\n_i (M_j^{-1}(\b^j)) > b_i.$ Because $\n_i (M_j^{-1} (\b^j (t)))$ is a continuous function on $t,$ 
%converges to $\n_i (M_j^{-1} (\b^j))$, 
there exists $\delta_j'$ such that 
%for any $0 < t < \delta_j',$ we have $\n_i (M_j^{-1}(\b^j(t))) > b_i.$ 
\begin{equation}\label{eq:deltaj'}
	0 < |t| < \delta_j' \quad \Rightarrow \quad \n_i (M_j^{-1}(\b^j(t))) > b_i.
\end{equation}

We choose 
\[\delta_\epsilon = \min(\{\delta_j \ | \ 1\le j \le k \}, \{\delta_j' \ | \ M_j^{-1}(\b^j) \not\in \vert(P)\}).\]

We claim for any $t: 0 < |t| < \delta_\epsilon,$ and any vertex $w$ of $P(t)$, there exists a vertex $v$ of $P$ such that $|w - v|< \epsilon.$ After all, since $w$ is a vertex of $P(t),$ for some (not necessarily unique) $j,$ we have that 
\begin{align}
	w =& \ M_j^{-1}(\b^j(t)) \\
	\n_i(M_j^{-1}(\b^j)) \le& \ b_i, \ \forall i \not\in I_j. \label{eq:contra}
\end{align}
Hence, $M_j^{-1}(\b^j)$ has to be a vertex of $P$ because otherwise $\delta_j'$ exists and \eqref{eq:deltaj'} contradicts with \eqref{eq:contra}.
Then our claim, i.e., (i), immediately follows from \eqref{eq:deltaj}.

We now prove (ii). %If we choose $\epsilon$ strictly smaller than one half of the minimum distances between two vertices of $P,$ 
The uniqueness choice of $v$ follows from the triangle inequality. 
%Finally, by our argument above, if $\n_i(\x) = b_i(t)$ is a supporting hyperplane for $w,$ one must have that $\n_i(\x) = b_i$ is a supporting hyperplane for $v.$ Therefore, we have \eqref{eq:Nsubset}.
Suppose $\n_i(\x) = b_i(t)$ is a supporting hyperplane for $w.$ Then $\exists j: 1 \le j \le k$ such that $\n_i$ is a row vector of $M_j$ and $w = M_j^{-1}(\b^j(t)).$ By our argument above, $M_j^{-1}(\b)$ is the unique vertex $v$ of $P$ that is within $\epsilon$-distance to $w.$ %However, such a vertex is unique. Thus, $v = M_j^{-1}(\b).$ 
Hence, $\n_i(\x) = b_i$ is a supporting hyperplane for $v=M_j^{-1}(\b).$ Therefore, \eqref{eq:Nsubset} follows. The inclusion relation \eqref{eq:dfcsubset} follows from \eqref{eq:Nsubset} and \eqref{equ:gendfcone}.

Finally, we prove (iii). Suppose $P(t)$ is non-empty and has vertices. It is clear that $P$ is a polyhedron with vertices, thus without lines. Also, $P$ and $P(t)$ have the same recession cone:
\[ K := \{ \x \in \R^D \ | \ A \x \le 0 \}.\]
By \eqref{eq:dfcsubset}, we have
\[ (\fcone(P(t),w))^\circ \subseteq (\fcone(P, v))^\circ , \quad \forall w \in W_{t,v}.\]
Thus,
\[ \bigcup_{w \in W_{t,v}} (\fcone(P(t),w))^\circ \subseteq (\fcone(P, v))^\circ .\]
	However, applying \eqref{equ:dfcone} to both $P$ and $P(t),$ we get
	\begin{align*}
		\sum_{v \in \vert(P)} [ (\fcone(P, v))^\circ ] \equiv [ K^\circ ]& \equiv \sum_{v \in \vert(P)} \sum_{w \in W_{t,v}} [ (\fcone(P(t),w)^\circ ] \\
		 & \text{ modulo polyhedra in proper subspaces.} 
	\end{align*}
	Hence, we must have
	\[ [ (\fcone(P, v))^\circ ] \equiv \sum_{w \in W_{t,v}} [ (\fcone(P(t),w)^\circ ]  \text{ modulo polyhedra in proper subspaces.} \] 
	Then (iii) follows from Lemma \ref{lem:equvipolar}, the above formula, and Theorem \ref{thm:rational}.
\end{proof}

\begin{proof}[Proof of Theorem \ref{thm:func}]
%It is clear that $P$ is a polyhedron with vertices, thus without lines. Also, $P$ and all the $P(t)$ have the same recession cone:
%\[ K := \{ \x \in \R^D \ | \ A \x \le 0 \}.\]
Clearly, (ii) follows from (i) and Lemma  \ref{lem:vertconv}/(iii). So we just need to prove (i).

Fix $j \in \{1, \dots, \ell\}.$ %It suffices to show that for sufficiently small $\epsilon$, there exists $\delta_\epsilon$ and a unique vertex $v$ of $P$ such that $|w_{t,j} - v| < \epsilon$ for any $t: |t| < \delta_\epsilon.$ However, 
	By Lemma \ref{lem:vertconv}, for sufficiently small $\epsilon$, there exists $\delta_\epsilon$, such that for any $t: 0 < |t| < \delta_\epsilon$, any $w_{t,j}$ is within $\epsilon$-distance to a unique vertex, say $v_t$, of $P.$ It suffices to show that $v_t$ are the same vertex for any $t: 0 < |t| < \delta_\epsilon.$ We assume to the contrary that there exist $t_1 \neq t_2$ such that $v_{t_1} \neq v_{t_2}.$ However, by \eqref{eq:dfcsubset}, we have
	\[ K_j^\circ \subseteq (\fcone(P,v_{t_1}))^\circ
	\text{ and } K_j^\circ \subseteq (\fcone(P,v_{t_2}))^\circ.\]
Hence, 	
	\[ K_j^\circ \subseteq (\fcone(P,v_{t_1}))^\circ \cap (\fcone(P,v_{t_2}))^\circ.\]
	Since $K_j$ is a pointed cone, the polar $K_j^\circ$ is full-dimensional. However, the intersection of $(\fcone(P,v_{t_1}))^\circ$ and $(\fcone(P,v_{t_2}))^\circ$ lies in the hyperplane $\{ \n \ | \ \n(v_{t_1}-v_{t_2}) =0\}$. This is a contradiction. Hence, we have (i).
%
	%By \eqref{eq:dfcsubset} again, we have
%\[ K_j^\circ \subseteq (\fcone(P, v))^\circ , \forall j \in J_v.\]
%Thus,
%\[ \bigcup_{j \in J_v} K_j^\circ \subseteq (\fcone(P, v))^\circ .\]
%	However, applying \eqref{equ:dfcone} to both $P(t)$ and $P,$ we get
%	\begin{eqnarray*}
%		\sum_{v \in \vert(P)} \sum_{j \in J_v} [ K_j^\circ ] \equiv &[ K^\circ ]& \equiv \sum_{v \in \vert(P)} [ (\fcone(P, v))^\circ ] \\
%		& & \text{ modulo polyhedra in proper subspaces.} 
%	\end{eqnarray*}
%	Hence, we must have
%	\[ \sum_{j \in J_v} [ K_j^\circ ] \equiv [ (\fcone(P, v))^\circ ] \text{ modulo polyhedra in proper subspaces.} \] 
%	Then (ii) follows from Lemma \ref{lem:equvipolar}, the above formula, and Theorem \ref{thm:rational}.
\end{proof}

%\begin{cor}\label{cor:funcentire}
%	We assume the same conditions as Theorem \ref{thm:func}, and assume further that $P$ is integral. Then 
%	\[ f(P, \z) = \sum_{v \in \vert(P)} \z^v \sum_{j \in J_v} f(K_j, \z) = \sum_{j=1}^\ell \z^{\lim_{t \to 0} w_{t,j}} f(K_j, \z).\] 
%\end{cor}

\begin{rem}\label{rem:seq}
	Theorem \ref{thm:func} and Theorem \ref{thm:funcentire} still hold if we replace the continuous function $\b(t)$ with a sequence $\{ \b_1, \b_2, \dots \}$ converging to $\b$.
\end{rem}

\section{Properties of Transportation Polytopes}\label{sec:proptrans}
In Section \ref{sec:unipert}, we will apply the perturbation method introduced in the last section to transportation polytopes.
Before that, we review properties of transportation polytopes. Most of them can be found in \cite{yemelichevKK}.

%We first note that the original interpretation of transportation polytope comes from a problem of transportation: Suppose there are $m$ source supply a product, and $n$ sink receive the product. The $i$th source supply $r_i$ units of the product and the $j$th sink receive $c_j$ units of the products. The each point $M$ in $\cT(\br, \bc)$ corresponds to one way to transport these products, where $M(i,j)$ is the number of units that is transported from the $i$th source to the $j$th sink.
We first note that the original interpretation of transportation polytope comes from planning the carriage of goods: There are $m$ suppliers which can supply the same product and which must be delivered to $n$ users. Suppose the $i$th supplier produces $r_i$ units of the product and the $j$th user requests $c_j$ units.
Then each point $M$ in $\cT(\br,\bc)$ corresponds to one way of distributing this product from the suppliers to the users, where $M(i,j)$ is the number of units that is transported from the $i$th supplier to the $j$th user. 

%With this interpretation, one sees that
Transportation polytopes have natural connection to the complete bipartite graphs. 
%Let $K_{m,n}$ be the complete bipartite graph with $m$ vertices on the left and $n$ vertices on the right. Denote by $e_{i,j}$ the edge in $K_{m,n}$ connecting the $i$th left vertex and the $j$th right vertex. 
Let $K_{m,n}$ be the complete bipartite graph with $m$ vertices $u_1, \dots, u_m$ on the left and $n$ vertices $w_1, \dots, w_n$ on the right. (One could consider that the $u_i$'s are suppliers and $w_j$'s are users.) Throughout the rest of the paper, we refer to $u_i$'s as the left vertices and $w_j$'s as the right vertices. Denote by $e_{i,j}$ the edge in $K_{m,n}$ connecting $u_i$ and the $w_j.$ For any subgraph $G$ of $K_{m,n},$ we denote by $E(G)$ the edge set of $G.$

Let $A_{m,n}$ be the $(m + n) \times mn$ incidence matrix of $K_{m,n}$. (Then the column $A_{m,n}^{i,j}$ of $A_{m,n}$ corresponding to the edge $e_{i,j}$ is the $(m+n)$-vector where the $i$th and $(m+j)$th component are $1$ and zero elsewhere.) Then the transportation polytope $\cT(\br, \bc)$ can be described by
\[ A_{m,n} \x = \begin{pmatrix}\br^T \\ \bc^T\end{pmatrix}, \ \ \x \ge 0.\]
	We often call matrix $A_{m,n}$ the {\it constraint matrix} of the transportation polytopes of order $m \times n.$

A vertex of a transportation polytope of order $m \times n$ is {\it non-degenerate} if it has exactly $m+n-1$ entries that are positive; otherwise it is {\it degenerate}. A transportation polytope is {\it non-degenerate} if all its vertices are non-degenerate; otherwise it is {\it degenerate}. It's easy to see that every non-degenerate transportation polytope is simple.
It is known that $A_{m,n}$ is a {\it totally unimodular matrix}, i.e., every minor of $A_{m,n}$ is $0$, $1$, or $-1$. Therefore, if a transportation polytope is non-degenerate or simple, it is a totally unimodular polytope. Then we can apply Corollary \ref{cor:totuni} to find its MGF.  
%For generic choices of $\br$ and $\bc,$ we have that $\cT(\br,\bc)$ is non-degenerate. So one would be able to find the MGF of a lot of transportation polytopes easily. 

There is an easy way to identify non-degenerate transportation polytopes.
\begin{thm}[Theorem 1.2 of Chapter 6 in \cite{yemelichevKK}]\label{thm:chardeg}
	The transportation polytope $\cT(\br,\bc)$ is non-degenerate if and only if the only nonempty index subsets $I \subseteq [m]$ and $J \subseteq [n]$ satisfying $\sum_{i \in I} r_i = \sum_{j \in J} c_j$ are $I = [m]$ and $J = [n].$
\end{thm}

To express the MGF of a totally unimodular polytope, we need to know how to describe its vertices, and the generating rays of each feasible cone.
%It will be convenient to identify the coefficients of an $m \times n$ matrix $M$ in $\cT(\br,\bc)$ with edge weights on the complete bipartite graph $K_{m,n}:$
%For any $M$ in $\cT(\br,\bc)$, we can consider it as a weight function $w$ on edge set of $K_{m,n}:$ 
%Let $M(i,j)$ be the weight $w(e_{i,j})$ of the edge $e_{i,j}$. Then we have
%\begin{equation}\label{equ:wt}
%	\sum_{k=1}^n w(e_{i,k}) = r_i, \forall i \ \text{and} \ \sum_{k=1}^m w(e_{k,j}) = c_j, \forall j.
%\end{equation}
%It is clear that points $M$ in $\cT(\br,\bc)$ is in one-to-one correspondence to weight functions \[ w: \text{edges of $K_{m,n}$} \to \R_{\ge 0}\] satisfying \eqref{equ:wt}. 

\subsection{Vertices of transportation polytopes} The vertices of a transportation polytope can be characterized with the {\it auxiliary graphs}.

\begin{defn}\label{defn:aux}
%	For any $M \in \cT(\br,\bc)$, we define $\supp(M)$ to be the set of edges of $K_{m,n}$ that have positive corresponding weights.
%	For any $M \in \cT(\br,\bc)$, we define $\supp(M)$ to be the set of indices $(i,j)$ such that $M(i,j)$ is positive.

	%The {\it auxiliary graph} of $M$, denoted by $\aux(M)$, is the subgraph of $K_{m,n}$ with edge set $\supp(M).$
	%The {\it auxiliary graph} of $M$, denoted by $\aux(M)$, is the subgraph of $K_{m,n}$ with edge set $\{ e_{i,j} \ | \ (i,j) \in \supp(M)\}.$
	The {\it auxiliary graph} of $M$, denoted by $\aux(M)$, is the subgraph of $K_{m,n}$ with edge set $\{ e_{i,j} \ | \ M(i,j) > 0 \}.$

%	We denote by $\Aux(\br,\bc)$ the set of all the auxiliary graphs obtained from points in $\cT(\br, \bc).$
	We denote by $\vertAux(\br,\bc)$ the set of all the auxiliary graphs obtained from vertices of $\cT(\br, \bc).$

%	For any subgraph $G$ of $K_{m,n}$, if there exists a weight function $w$ that assigns positive numbers to the edges in $G$ and assigns zeros to the other edges in $K_{m,n}$ satisfying \eqref{equ:wt}, we say $G$ 
\end{defn}

A point $M \in \cT(\br, \bc)$ is a vertex of $\cT(\br,\bc)$ if and only if the column vectors $\{ A_{m,n}^{i,j} \ | \ M(i,j) > 0 \}$ are linearly independent. However, using matroid theory, one can show that given any index set $ind \subset [m] \times [n],$ the column vectors $\{ A_{m,n}^{i,j} \ | (i,j) \in ind\}$ are linearly independent if and only if the subgraph of $K_{m,n}$ with edge set $\{ e_{i,j} \ | \ (i,j) \in ind\}$ is a forest. Therefore, we have the following theorem and lemma, both of which follow from Theorem 2.2 of Chapter 6 in \cite{yemelichevKK}.

\begin{thm}%[Theorem 2.2 of Chapter 6 in \cite{yemelichevKK}]
	\label{thm:charvert}
	Let $M \in \cT(\br, \bc).$ Then $M$ is a vertex of $\cT(\br,\bc)$ if and only if $\aux(M)$ is a spanning forest of $K_{m,n}.$ Furthermore, $\aux$ induces a bijection between the vertex set of $\cT(\br, \bc)$ and the set of spanning forests of $K_{m,n}$ that are auxiliary graphs of some points in $\cT(\br, \bc).$

	In particular, if $\cT(\br,\bc)$ is non-degenerate, $M$ is a vertex of $\cT(\br,\bc)$ if and only if $\aux(M)$ is a spanning tree of $K_{m,n}.$ Thus, $\aux$ induces a bijection between the vertex set of $\cT(\br, \bc)$ and the set of spanning trees of $K_{m,n}$ that are auxiliary graphs of some points in $\cT(\br, \bc).$

\end{thm}

\begin{lem}\label{lem:samevrt}
	Let $M$ and $N$ be two vertices in $\cT(\br, \bc).$ Suppose there is a spanning tree $T$ of $K_{m,n}$ such that both $\aux(M)$ and $\aux(N)$ are subgraphs of $T.$ Then $M$ and $N$ are the same vertices. 
\end{lem}

\begin{ex}\label{ex:aux}
Let $\br = \bc = (1, 1, 2).$ The transportation polytope $\cT(\br,\bc)$ have seven vertices:
\begin{align*}
	& M_0 = \begin{pmatrix}1&0&0 \\ 0&1&0 \\ 0&0&2 \end{pmatrix}, \ 
	  M_1 = \begin{pmatrix}0&1&0 \\ 1&0&0 \\ 0&0&2 \end{pmatrix}, \
	  M_2 = \begin{pmatrix}1&0&0 \\ 0&0&1 \\ 0&1&1 \end{pmatrix}, \
	  M_3 = \begin{pmatrix}0&0&1 \\ 0&1&0 \\ 1&0&1 \end{pmatrix}, \\
	&  M_4 = \begin{pmatrix}0&1&0 \\ 0&0&1 \\ 1&0&1 \end{pmatrix}, \
	  M_5 = \begin{pmatrix}0&0&1 \\ 1&0&0 \\ 0&1&1 \end{pmatrix}, \
	  M_6 = \begin{pmatrix}0&0&1 \\ 0&0&1 \\ 1&1&0 \end{pmatrix}. 
\end{align*}
Let $T_i = \aux(M_i)$ be the auxiliary graph of $M_i$ for $0 \le i \le 6.$ Then \begin{align*}
	& E(T_0) = \{ e_{1,1}, e_{2,2}, e_{3,3}\}, \
	 E(T_1) = \{ e_{1,2}, e_{2,1}, e_{3,3}\}, \
	 E(T_2) = \{ e_{1,1}, e_{2,3}, e_{3,2}, e_{3,3}\}, \\
	& E(T_3) = \{ e_{1,3}, e_{2,2}, e_{3,1}, e_{3,3}\}, \ 
	 E(T_4) = \{ e_{1,2}, e_{2,3}, e_{3,1}, e_{3,3}\}, \\ 
	& E(T_5) = \{ e_{1,3}, e_{2,1}, e_{3,2}, e_{3,3}\}, \ 
	 E(T_6) = \{ e_{1,3}, e_{2,3}, e_{3,1}, e_{3,2}\}. 
\end{align*}
Each $T_i$ is a spanning forest of $K_{3,3}.$
\end{ex}

We remark that a weak statement of Theorem \ref{thm:charvert} is that $\aux$ induces a bijection between $\vert(\cT(\br,\bc))$ and $\vertAux(\br,\bc).$ This is the first reason why the set $\vertAux(\br,\bc)$ is important (and why we have defined this set). In fact, in addition to the one-to-one correspondence to the vertex set of $\cT(\br,\bc),$ the set $\vertAux(\br,\bc)$ also determines the combinatorial structure of $\cT(\br,\bc)$ as we will see in the next subsection (cf. Corollary \ref{cor:samefeas}).  

%The following lemma characterizes the sets of rays that generates the feasible cone of vertices of a non-degenerate transportation polytope. One can modify the lemma for general transportation polytopes; but it requires more definitions. Therefore, for simplicity, we only state the version for non-degenerate cases. 

\subsection{Feasible cones of transportation polytopes} 
Section 4.1 of Chapter 6 in \cite{yemelichevKK} gives a complete description of the generating rays of feasible cones of transportation polytopes, as well as a characterization of when two vertices are adjacent in a transportation polytope. We summarize the results as two lemmas below (Lemma \ref{lem:charray} and Lemma \ref{lem:adjvert}). We note that although the results in \cite{yemelichevKK} are presented slightly differently, our lemmas follow easily from them. We begin with a preliminary definition.

\begin{defn}\label{defn:cyc}
	Let $T$ be a spanning forest of $K_{m,n}$. Suppose $\be =\{e^{(1)}, e^{(2)}, \dots, e^{(k)}\}$ is a set of distinct edges that are not in $T$ satisfying the following conditions:
	\begin{alist}
	\itm $T \cup \be$ creates a unique cycle. %Note this cycle is unique if it exists. 
	(Note that this cycle must have even length.)
	\itm All the edges in the set $\be$ are contained in the cycle. 
	\itm Suppose this cycle is:
	\[ e^{(1)} = e_{i_1, j_1}, e_{i_2, j_1}, e_{i_2, j_2}, \dots, e_{i_s,j_s}, e_{i_1, j_s}. \]
	Then each $e^{(\ell)}$ appears at an odd position in the above list. This is equivalent to say that any two edges in the set $\be$ have an even distance in the cycle.
\end{alist}
	Given such a pair $T$ and $\be,$
	we define $\cyc(T, \be)$ to be the $m \times n$ matrix whose entries are defined as
	\begin{equation}\label{equ:cyc}
	\cyc(T,\be)(i,j)  = \begin{cases}1, & \text{if $(i,j) \in \{(i_1,j_1), (i_2,j_2), \dots, (i_s, j_s)\}$};\\
		-1, & \text{if $(i,j) \in \{(i_2,j_1), (i_3, j_2), \dots, (i_1, j_s)\}$};\\
        0, &  \mbox{otherwise}.\end{cases}
\end{equation}

\end{defn}

\begin{ex}\label{ex:cycle}
	Let $T$ be the spanning forest $T_0$ of $K_{3,3}$ defined in Example \ref{ex:aux} and $\be = \{e_{1,2}, e_{2,1}\}.$ Then the pair $(T, \be)$ satisfies the conditions in Definition \ref{defn:cyc}. In particular, the unique cycle in $T \cup \be$ is: $e_{1,2}, e_{2,2}, e_{2,1}, e_{1,1}.$ Hence, 
	\[ \cyc(T, \be) = \begin{pmatrix} -1&1&0 \\ 1&-1&0 \\ 0&0&0 \end{pmatrix}.\]
\end{ex}

\begin{lem}\label{lem:charray}

	Let $M$ be a vertex in $\cT(\br, \bc),$ and $T = \aux(M)$ the auxiliary graph of $M.$ (By Theorem \ref{thm:charvert}, $T$ is a spanning forest of $K_{m,n}.$) Then 
	\begin{equation}\label{equ:rays}
		\{ \cyc(T,\be) \ | \ \text{$(T, \be)$ is a pair satisfies the conditions in Definition \ref{defn:cyc} }\}
	\end{equation}
is the set of rays that generates the feasible cone of $\cT(\br, \bc)$ at the vertex $M.$
	
In particular, if $\cT(\br, \bc)$ is non-degenerate. Then the generating set \eqref{equ:rays} can be rewritten as
	\begin{equation}\label{equ:raysnondeg}
		\{ \cyc(T,e) \ | \ e \not\in E(T)\}.
	\end{equation}
\end{lem}

\begin{lem}\label{lem:adjvert}
	Let $M$ and $N$ be two distinct vertices in $\cT(\br, \bc).$ Then $M$ and $N$ are adjacent if and only if the union of $\aux(M)$ and $\aux(N)$ has a unique cycle. 
	
	Moreover, if $M$ and $N$ are two adjacent vertices, the unique cycle in the union of $\aux(M)$ and $\aux(N)$ is the same as $\aux(M) \cup \be,$ for some $\be =\{e^{(1)}, e^{(2)}, \dots, e^{(k)}\}$ satisfying the conditions in Definition \ref{defn:cyc}. %$\cyc(\aux(M), e),$ for some $e \not\in E(\aux(M)).$ 
	Thus, the unique cycle determines the ray from $M$ to $N$ as described in \eqref{equ:cyc}.
\end{lem}

\begin{ex}\label{ex:notsimple}
	Let $\cT(\br,\bc)$ and $M_i$ and $T_i$ be the ones defined in Example \ref{ex:aux}. For each $1 \le i \le 5,$ the union of $T_0$ and $T_i$ has a unique cycle; but the union of $T_0$ and $T_6$ has more than one cycles. Therefore, the vertex $M_0$ has five adjacent vertices $M_1,$ $M_2,$ \dots, $M_5.$ For each $1 \le i \le 5,$ the unique cycle in the union of $T_0$ and $T_i$ determines one of the rays that generates the feasible cone of $\cT(\br,\bc)$ at the vertex $M_0.$ 

	The unique cycle in the union of $T_0$ and $T_1$ is the cycle described in Example \ref{ex:cycle}. Hence, we obtain one ray
	\[ R_1 = \cyc(T_0, \{e_{1,2}, e_{2,1}\}) = \begin{pmatrix} -1&1&0 \\ 1&-1&0 \\ 0&0&0 \end{pmatrix}.\]
	We can similarly find the other four rays, denoting by $R_i$ the ray determined by the union of $T_0$ and $T_i:$
	\begin{align*} 
		R_2 =& \cyc(T_0, \{e_{2,3}, e_{3,2}\}) = \begin{pmatrix} 0&0&0 \\ 0&-1&1 \\ 0&1&-1 \end{pmatrix}; \\
		R_3 =& \cyc(T_0, \{e_{1,3}, e_{3,1}\}) = \begin{pmatrix} -1&0&1 \\ 0&0&0 \\ 1&0&-1 \end{pmatrix}; \\
			R_4 =& \cyc(T_0, \{e_{1,2}, e_{2,3}, e_{3,1}\}) = \begin{pmatrix} -1&1&0 \\ 0&-1&1 \\ 1&0&-1 \end{pmatrix}; \\
				R_5 =& \cyc(T_0, \{e_{1,3}, e_{2,1}, e_{3,2}\}) = \begin{pmatrix} -1&0&1 \\ 1&-1&0 \\ 0&1&-1 \end{pmatrix}. 
		\end{align*}
Thus, $\fcone(\cT(\br,\bc), M_0)$ is generated by $R_1, R_2,$ \dots, $R_5.$ Note that this is not a simple cone.
\end{ex}

It turns out that the transportation polytope $\cT(\br,\bc)$ is integral if $\br$ and $\bc$ are both integer vectors. Hence, the following corollary follows immediately from Lemma \ref{lem:unicone}, Corollary \ref{cor:totuni} and Lemma \ref{lem:charray}. 

\begin{cor}\label{cor:mgfnondeg}
	Suppose $\cT(\br,\bc)$ is a non-degenerate transportation polytope. For any vertex $M$ of $\cT(\br,\bc),$ we have
\begin{equation}\label{equ:mgfnondegvert}
	f(\fcone(\cT(\br, \bc),M),\z) =  \prod_{e \not\in E(\aux(M))} \frac{1}{1 - \z^{\cyc(\aux(M),e)}}.
\end{equation}

%Suppose $\br$ and $\bc$ are integer vectors and $\cT(\br, \bc)$ a non-degenerate transportation polytope. Then the multivariate generating function of $\cT(\br, \bc)$ is  
Suppose further that $\br$ and $\bc$ are integer vectors. Then the multivariate generating function of $\cT(\br, \bc)$ is  
\begin{equation}\label{equ:mgfnondeg}
	f(\cT(\br, \bc),\z) = \sum_{M \in \vert(\cT(\br, \bc))} \z^M \prod_{e \not\in E(\aux(M))} \frac{1}{1 - \z^{\cyc(\aux(M),e)}}.
\end{equation}
\end{cor}

%	It turns out that two distinct vertices of a transportation polytope are adjacent if and only if the union of their auxiliary graphs has a unique cycle.

We also have the following corollary to Lemma \ref{lem:adjvert}.
\begin{cor}\label{cor:samefeas}
Suppose $\cT(\br, \bc)$ and $\cT(\br', \bc')$ are two transportation polytopes of order $m \times n$ satisfying \[ \vertAux(\br, \bc) = \vertAux(\br', \bc').\] Let $T \in \vertAux(\br, \bc)$, and $M_T$ and $M_T'$ be the vertices of $\cT(\br,\bc)$ and $\cT(\br',\bc')$, respectively, corresponding to $T$. Then 
\[ \fcone(\cT(\br,\bc), M_T) = \fcone(\cT(\br', \bc'), M_T').\]
\end{cor}

\section{A perturbation for transportation polytopes}\label{sec:unipert} We describe a perturbation that works for any transportation polytope.

\begin{lem}\label{lem:unipert}
Suppose $\br = (r_1, r_2, \dots, r_m)$ and $\bc = (c_1, c_2, \dots, c_n)$ are two rational positive vectors. We define
\[ \br(t) = (r_1 - t, \dots, r_m - t), \ \bc(t) = (c_1, \dots, c_{n-1}, c_n-mt),\ 0 \le t < \frac{1}{K m},\]
 where $K$ is the lowest common multiple of the denominators of $\br$ and $\bc.$ Then we have the following:%for any $t \in (0, \frac{1}{K m})$, the transportation polytope $\cT(\br(t), \bc(t))$ is non-degenerate.
 \begin{ilist}
\itm For any $t \in (0, \frac{1}{K m})$, the transportation polytope $\cT(\br(t), \bc(t))$ is non-degenerate. 
\itm%$\{ \cT(\br(t), \bc(t)) \ | \ t \in (0, \frac{1}{K m})\}$ is a family of non-degenerate central transportation polytopes satisfying the condition of Lemma \ref{lem:func}. 
For any $t \in (0, \frac{1}{K m}),$ the set $\vertAux(\br(t),\bc(t))$ is independent of $t.$
\itm $\{ \cT(\br(t), \bc(t)) \ | \ t \in (0, \frac{1}{K m})\}$ is a family of transportation polytopes satisfying the condition of Theorem \ref{thm:func}.
\end{ilist}
\end{lem}

\begin{rem}
	The perturbation we defined in Lemma \ref{lem:unipert} perturbs the original polytope by $t (-1, -1, \dots, -1, 0, \dots, 0, -m)^T.$ The vector $(-1, -1, \dots, -1, 0, \dots, 0, -m)^T$ is in the cone spanned by the column vectors of the constraint matrix $A_{m,n}.$ We expect that for any generic vector $\bp$ in the cone spanned by the columns of $A_{m,n},$ perturbing the original polytope by $t \bp$ works. However, we choose this particular one since it gives nice combinatorial results for the central transportation polytopes of order $kn \times n$, which we will discuss in the next section. 
\end{rem}

We give the following two definitions before the proof of the above lemma.
\begin{defn}
	Let $x$ be a real number. The {\it floor} or {\it integer part} of $x$, denoted by $\lfloor x \rfloor$ is the biggest integer that is not greater than $x,$ and the {\it fractional part} of $x$, denoted by $\fract(x),$ is $x - \lfloor x \rfloor.$
	The {\it ceiling} of $x$, denoted by $\lceil x \rceil$ is the smallest integer that is not smaller than $x,$ and the {\it co-fractional part} of $x,$ denoted by $\cofrac(x),$ is $\lceil x \rceil - x.$
\end{defn}

\begin{defn}\label{defn:degree}
	Let $G$ be a subgraph of the complete bipartite graph $K_{m,n}.$ 
	
	Let $d_j$ be the number of edges in $G$ connecting to $w_j,$ the $j$th right vertex of $K_{m,n}.$ We call $(d_1, d_2, \dots, d_n)$ the {\it right degree sequence} of $G$ and write $\delta(G) = (d_1, \dots, d_n).$

%	We also define \[ r(G) := \text{the number of $w_j$'s in $G$, i.e., the number of right vertices in $G$}.\]
	
	We also define \[ l(G) := \text{the number of $u_i$'s in $G$, i.e., the number of left vertices in $G$}.\]

	For convenience, for any spanning tree $T$ of $K_{m,n},$ we consider $T$ as a rooted tree rooted at $w_n,$ the $n$th right vertex. For any vertex $v$ of $T,$ we denote by $T_v$ the subtree of $T$ rooted at $v.$ 
\end{defn}

\begin{proof}[Proof of Lemma \ref{lem:unipert}]
Because $\cT(K\br(t), K\bc(t))$ is a dilation of $\cT(\br(t), \bc(t)),$ the polytopes $\cT(K\br(t), K\bc(t))$ and $\cT(\br(t), \bc(t))$ have exactly the same combinatorial structures. Therefore, without loss of generality, we can assume that $\br$ and $\bc$ are integer vectors, and $K = 1.$

\begin{ilist}
\itm Suppose $\emptyset \neq I \subseteq [m]$ and $\emptyset \neq J \subseteq [n]$ are two index sets satisfying 
\begin{equation}\label{equ:pfnd}
	\sum_{i \in I} \br_i(t) = \sum_{j \in J} \bc_j(t).
\end{equation}
The co-fractional part of the left hand side of \eqref{equ:pfnd} is $|I| t \neq 0.$ Hence, $n \in J,$ because otherwise the right hand side of \eqref{equ:pfnd} is an integer. Then the co-fractional part of the right hand side of \eqref{equ:pfnd} is $mt.$ Therefore, $|I| = m$ and $I = [m].$ This implies that $J = [n].$
Therefore, by Theorem \ref{thm:chardeg}, the polytope $\cT(\br(t), \bc(t))$ is non-degenerate.

\itm Let $t_0 \in (0, \frac{1}{Km}=\frac{1}{m}),$ and let $T \in \vertAux(\br(t_0), \bc(t_0)).$ By Theorem \ref{thm:charvert}, $T$ is a spanning tree of $K_{m,n}.$ It suffices to show that for any $t \in (0, \frac{1}{m}),$ the tree $T$ is the auxiliary graph of a vertex of $\cT(\br(t), \bc(t)).$ Let $M_T(t_0)$ be the vertex of $\cT(\br(t_0), \bc(t_0))$ corresponding to the tree $T.$ 

We claim that, considering $T$ a rooted tree rooted at $w_n,$
\begin{align*}
	\cofrac(M_T(t_0)(i,j)) =& \ l(T_{u_i}) \ t_0, \quad \text{ if $w_j$ is the parent of $u_i$ in $T$;} \\
	\fract(M_T(t_0)(i,j)) =& \ l(T_{w_j}) \ t_0, \quad \text{ if $u_i$ is the parent of $w_j$ in $T$;}
\end{align*}
and the matrix $M_T$ whose entries defined by the following equation is a vertex of $\cT(\br, \bc).$
\begin{equation}\label{equ:MT}
	M_T(i,j)  = \begin{cases}\lceil M_T(t_0)(i,j) \rceil, & \text{ if $w_j$ is the parent of $u_i$ in $T$;} \\
		\lfloor M_T(t_0)(i,j) \rfloor, & \text{ if $u_i$ is the parent of $w_j$ in $T$;} \\
        0, &  \mbox{otherwise}.\end{cases}
\end{equation}
The claim can be proved by induction on hook lengths of vertices of $T.$ (Recall the {\it hook length} of a vertex $v$ in a rooted tree is the number of descendants of $v.$) We define the matrix $M_T(t)$ as follows:
	\begin{equation}\label{equ:MTt}
	M_T(t)(i,j)  = \begin{cases} M_T(i,j) - l(T_{u_i}) \ t, & \text{ if $w_j$ is the parent of $u_i$ in $T$;} \\
		M_T(i,j) +  l(T_{w_j}) \ t, & \text{ if $u_i$ is the parent of $w_j$ in $T$;} \\
        0, &  \mbox{otherwise}.\end{cases}
\end{equation}
It is clear that $M_T(t)$ is a vertex of $\cT(\br(t), \bc(t))$ with auxiliary graph $T.$

See Figure \ref{fig:exvalue} for an example. The tree on the right side of the figure is the same one as the one on the left; it is just drawn as it is a rooted tree rooted at $w_3.$ The solid lines are the edges where some $w_j$ is the parent and the dashed lines are the edges where some $u_i$ is the parent. The number next to an edge $\{u_i,w_j\}$ is the entry of $M_T(t)(i,j).$
\begin{figure}[htb]
\centering
\scalebox{0.9}{\input{exvalue.pstex_t}}
\caption{}
\label{fig:exvalue}
\end{figure}

\itm It follows from (ii) and Corollary \ref{cor:samefeas}.
\end{ilist}
\end{proof}

It turns out we can use auxiliary graphs to determine which vertices of the perturbed polytopes $\cT(\br(t), \bc(t))$ (described in Lemma \ref{lem:unipert}) converge to a given vertex of $\cT(\br, \bc).$
\begin{lem}\label{lem:uniquevert}
	Assume the conditions of Lemma \ref{lem:unipert}. Let $t_0 \in (0, \frac{1}{Km}),$ and let $T \in \vertAux(\br(t_0),\bc(t_0)).$ Define $M_T$ and $M_T(t)$ as in \eqref{equ:MT} and \eqref{equ:MTt}, respectively. Then
	\begin{ilist}
		\itm $\lim_{t \to 0} M_T(t) = M_T.$
		\itm $M_T$ is the unique vertex of $\cT(\br,\bc)$ satisfying that its auxiliary graph is a subgraph of $T.$
	\end{ilist}
\end{lem}

\begin{proof}
	It is clear that $M_T(t)$ converges to $M_T,$ and $M_T$ is a vertex of $
	\cT(\br,\bc)$ satisfying that its auxiliary graph is contained in $T.$ It is left to show the uniqueness. However, the uniqueness follows from Lemma \ref{lem:samevrt}.
\end{proof}

\begin{ex}\label{ex:perturb}
Let $\br=\bc=(1,1,2)$ and $M_0, M_1, \dots, M_6$ be the vertices of $\cT(\br,\bc)$ as defined in Example \ref{ex:aux}.
	Then we define
	\[ \br(t) = (1-t, 1-t, 2-t), \ \bc(t) = (1, 1, 2-3t), \ 0 \le t < \frac{1}{3}.\]
	One can check that $\cT(\br(t), \bc(t))$ have $18$ vertices, which we list in the following table.

\[	\begin{tabular}{|c||c|}
		\hline
		& \\
	$M_i$ & Vertices of $\cT(\br(t), \bc(t))$ that converge to $M_i$ as $t \to 0$ \\[2mm]
	\hline
	$M_0$ & $\setlength{\arraycolsep}{2pt}\renewcommand{\arraystretch}{0.9}\begin{pmatrix}1-2t&t&0 \\ 0&1-t&0 \\ 2t&0&2-3t \end{pmatrix}, \begin{pmatrix}1-t&0&0 \\ t&1-2t&0 \\ 0&2t&2-3t \end{pmatrix}, \begin{pmatrix}1-t&0&0 \\ 0&1-t&0 \\ t&t&2-3t \end{pmatrix}$ \\[5mm]
	$M_1$ & $\setlength{\arraycolsep}{2pt}\renewcommand{\arraystretch}{0.9}\begin{pmatrix}t&1-2t&0 \\ 1-t&0&0 \\ 0&2t&2-3t \end{pmatrix}, \begin{pmatrix}0&1-t&0 \\ 1-2t&t&0 \\ 2t&0&2-3t \end{pmatrix}, \begin{pmatrix}0&1-t&0 \\ 1-t&0&0 \\ t&t&2-3t \end{pmatrix}$ \\[5mm]
	$M_2$ & $\setlength{\arraycolsep}{2pt}\renewcommand{\arraystretch}{0.9}\begin{pmatrix}1-t&0&0 \\ t&0&1-2t \\ 0&1&1-t \end{pmatrix}, \begin{pmatrix}1-t&0&0 \\ 0&0&1-t \\ t&1&1-2t \end{pmatrix}$ \\[5mm]
	$M_3$ & $\setlength{\arraycolsep}{2pt}\renewcommand{\arraystretch}{0.9}\begin{pmatrix} 0&t&1-2t \\ 0&1-t&0 \\ 1&0&1-t \end{pmatrix}, \begin{pmatrix}0&0&1-t \\ 0&1-t&0 \\ 1&t&1-2t \end{pmatrix}$ \\[5mm]
	$M_4$ & $\setlength{\arraycolsep}{2pt}\renewcommand{\arraystretch}{0.9}\begin{pmatrix}0&1-t&0 \\ 0&t&1-2t \\ 1&0&1-t \end{pmatrix}, \begin{pmatrix}0&1-t&0 \\ 0&0&1-t \\ 1&t&1-2t \end{pmatrix}$ \\[5mm]
	$M_5$ & $\setlength{\arraycolsep}{2pt}\renewcommand{\arraystretch}{0.9}\begin{pmatrix} t&0&1-2t \\ 1-t&0&0 \\ 0&1&1-t \end{pmatrix}, \begin{pmatrix}0&0&1-t \\ 1-t&0&0 \\ t&1&1-2t \end{pmatrix}$ \\[5mm]
		$M_6$ & $\setlength{\arraycolsep}{2pt}\renewcommand{\arraystretch}{0.9}\begin{pmatrix}0&0&1-t \\ 0&t&1-2t \\ 1&1-t&0 \end{pmatrix}, \begin{pmatrix} 0&t&1-2t \\ 0&0&1-t \\ 1&1-t&0 \end{pmatrix}, \begin{pmatrix}0&0&1-t \\ t&0&1-2t \\ 1-t&1&0 \end{pmatrix}, \begin{pmatrix} t&0&1-2t \\ 0&0&1-t \\ 1-t&1&0 \end{pmatrix}$ \\
			\hline
	\end{tabular}
	\]
The auxiliary graph of each of the 18 vertices gives a spanning tree of $K_{3,3}.$ It is clear that $\vertAux(\br(t), \bc(t)),$ the set of these 18 spanning trees, is independent of $t.$

One can tell which vertices of $\cT(\br(t), \bc(t))$ converge to $M_i$ directly from the description of the vertices. However, we can also determine this using Lemma \ref{lem:uniquevert}/(ii). For example, the auxiliary graph of $M_0$ is $T_0$ whose edge set is \[E(T_0) = \{ e_{1,1}, e_{2,2}, e_{3,3}\}.\] %It is clear that the auxiliary graphs of the three vertices of $\cT(\br(t), \bc(t))$ that converge to $M_0$ have to contain $T_0.$ However, if one checks the auxiliary graphs of the other 15 vertices of $\cT(\br(t), \bc(t)),$ neither of them contains $T_0.$
If we examine the 18 spanning trees in $\vertAux(\br(t), \bc(t)),$ three of them contains $T_0.$ The edge sets of these three trees are
\[ \{e_{1,1}, e_{1,2}, e_{2,2}, e_{3,1}, e_{3,3}\}, \{e_{1,1}, e_{2,1}, e_{2,2},e_{3,2}, e_{3,3}\}, \{e_{1,1}, e_{2,2}, e_{3,1}, e_{3,2}, e_{3,3}\}. \]
These are precisely the auxiliary graphs of the three vertices that converge to $M_0.$ 

\end{ex}

Lemma \ref{lem:uniquevert} tells us which vertices of the perturbed transportation polytopes $\cT(\br(t),\bc(t))$ converge to a given vertex $M$ of the original transportation polytope $\cT(\br,\bc):$ 
\[ \lim_{t \to 0} M_T(t) = M \ \Longleftrightarrow \ \text{$\aux(M)$ is a subgraph of $T$}.\]
Now we can describe the MGF of the feasible cone of each vertex of $\cT(\br,\bc)$
\begin{cor}\label{cor:PertAux}
Assume the conditions of Lemma \ref{lem:unipert}. Let $t \in (0, \frac{1}{Km}).$
For any vertex $M$ of $\cT(\br, \bc),$ let
\begin{equation}\label{equ:pertaux}
	\PertAux(M) := \{ T \in \vertAux(\br(t), \bc(t)) \ | \ \aux(M) \text{ is a subgraph of } T\}
\end{equation}
be the set of auxiliary graphs of vertices of the perturbed transportation polytope $\cT(\br(t),\bc(t))$ that contain the auxiliary graph of $M$ as a subgraph. 
	
Then
	\begin{align}
		[ \fcone(\cT(\br,\bc), M) ] \equiv& \sum_{T \in \PertAux(M)} [ \fcone(\cT(\br(t),\bc(t)), M_T(t))] \label{equ:inddeg} \\
		 & \quad \quad \quad \quad \quad \quad \text{ modulo polyhedra with lines,} \nonumber 
	\end{align}
	where $M_T(t)$ is defined as in \eqref{equ:MTt}.

Hence, 
\begin{equation}\label{equ:mgffeas} f(\fcone(\cT(\br,\bc),M), \z) =\sum_{T \in \PertAux(M)} \prod_{e \not\in E(T)} \frac{1}{1 - \z^{\cyc(T,e)}}. 
\end{equation}
\end{cor}

\begin{proof}
	Formula \eqref{equ:inddeg} follows from Lemma \ref{lem:unipert}/(iii), Theorem \ref{thm:func}, and Lemma \ref{lem:uniquevert}. 

	Formula \eqref{equ:mgffeas} follows from \eqref{equ:inddeg}, Lemma \ref{lem:unipert}/(i), and Formula \eqref{equ:mgfnondegvert}. 
	%To prove Formula \eqref{equ:inddeg}, it is sufficient to check that $\PertAux(M)$ is the set of all the $T$ such that $M_T(t)$ converges to $M.$ By the proof of Lemma \ref{lem:unipert}, it's clear that if $M_T(t)$ converges to $M$, the auxiliary graph of $M$ is a subgraph of $T.$ Now suppose $T \in \vertAux(\cT(\br(t_0), \bc(t_0)))$ such that $\aux(M)$ is a subgraph of $T.$ Suppose $M_T(t)$ converges to a vertex $N$ of $\cT(\br,\bc).$ Then $\aux(N)$ is a subgraph of $T.$ However, by Lemma \ref{lem:samevrt}, $M$ and $N$ are the same vertex. Therefore, Formula \eqref{equ:inddeg} follows. 
\end{proof}

\begin{ex}
	We assume the same setup as in Examples \ref{ex:aux}, \ref{ex:notsimple} and \ref{ex:perturb}. As we discussed in Example \ref{ex:notsimple}, the feasible cone of $\cT(\br,\bc)$ at $M_0$ is spanned by five rays and is not a simple cone. Therefore, it is hard to compute its MGF directly. Applying the perturbation we discuss in this section, as shown in Example \ref{ex:perturb}, we see that there are three vertices of the perturbed transportation polytope $\cT(\br(t), \bc(t))$ that converge to $M_0.$ Let $K_1, K_2$ and $K_3$ be the feasible cones of $\cT(\br(t), \bc(t)$ at these three vertices. Then \eqref{equ:inddeg} says that 
		\[ [ \fcone(\cT(\br,\bc), M_0) ] \equiv [K_1] + [K_2] + [K_3], \quad \text{ modulo polyhedra with lines.}  \] 
Hence,
\[ f(\fcone(\cT(\br,\bc), M_0), \z) = f(K_1, \z) + f(K_2, \z) + f(K_3, \z).\]
Since $\cT(\br(t),\bc(t))$ is non-degenerate, each $K_i$ is unimodular and one can obtain a formula for $f(K_i, \z)$ quickly from the set of generating rays of $K_i$. 
	Knowing the set $\vertAux(\br(t), \bc(t))$ of the auxiliary graphs of the vertices of $\cT(\br(t), \bc(t))$, we are able to figure out the rays that generate the feasible cone of $\cT(\br(t), \bc(t))$ at each of its vertices using Lemma \ref{lem:charray} for non-degenerate polytopes. In particular, we are able to describe feasible cones $K_1, K_2$ and $K_3,$ and then obtain a formula for $f(\fcone(\cT(\br,\bc), M_0), \z).$
\end{ex}

We can also describe the MGF of an integral transportation polytope.
\begin{cor}\label{cor:unipert}
	Assume the conditions of Lemma \ref{lem:unipert} and further assume that $\br$ and $\bc$ are integer vectors. Let $t \in (0, \frac{1}{m}).$
Then the multivariate generating function of $\cT(\br, \bc)$ is  
\begin{equation}\label{equ:mgfdeg}
	f(\cT(\br, \bc),\z) =\sum_{T \in \vertAux(\br(t), \bc(t))} \z^{M_T} \prod_{e \not\in E(T)} \frac{1}{1 - \z^{\cyc(T,e)}}, 
\end{equation}
where $M_T$ is defined as in \eqref{equ:MT}, or equivalently, $M_T$ is the unique vertex of $\cT(\br,\bc)$ satisfying that its auxiliary graph is a subgraph of $T.$
\end{cor}

\begin{proof}
	Formula \eqref{equ:mgfdeg} follows immediately from Lemma \ref{lem:unipert}, Theorem \ref{thm:funcentire}, and Lemma \ref{lem:uniquevert}.
\end{proof}

%Formula \eqref{equ:mgfdeg} gives the MGF of the entire central transportation polytope $\cT(\br,\bc)$. If we want to find the MGF of the feasible cone of each vertex $M$ of $\cT(\br,\bc),$ we need to describe the set of vertices in the perturbed polytopes $\cT(\br(t),\bc(t))$ that converges to $M.$ It turns out it can be done by looking at the auxiliary graphs again. 

\subsection*{Maximum number of vertices}
We finish this section by an additional result we obtain from the perturbation we define in Lemma \ref{lem:unipert}.

Suppose $\cT(\br,\bc)$ is a central transportation polytope of order $m \times n,$ and $\cT(\br',\bc')$ is a transportation polytope of same order. By Theorem 7.1 of Chapter 6 in \cite{yemelichevKK}, $\cT(\br',\bc')$ has the maximum possible number of vertices if and only if for all $\lambda \in (0,1)$, the transportation polytope $\cT(\lambda\br+(1-\lambda)\br', \lambda \bc + (1-\lambda)\bc')$ is non-degenerate. Because the perturbation we define in Lemma \ref{lem:unipert} is linear, we have the following result:

\begin{lem}\label{lem:maxvert}
	Suppose $\cT(\br,\bc)$ is a central transportation polytope. Then the transportation polytopes $\cT(\br(t), \bc(t))$, $0 < t < \frac{1}{Km},$ we defined in Lemma \ref{lem:unipert} achieve the maximum number of vertices among all the transportation polytopes of order $m \times n.$
\end{lem}

\begin{proof}
For any $\lambda \in (0,1),$ it is clear that
\[\cT(\lambda\br+(1-\lambda)\br(t), \lambda \bc + (1-\lambda)\bc(t)) = \cT(\br( (1-\lambda) t), \bc( (1-\lambda)t)).\]
Since $0 < (1-\lambda)t < t < \frac{1}{Km},$ by Lemma \ref{lem:unipert}/(i), the above transportation polytope is non-degenerate. Then the result follows.
\end{proof}

%We finish this section by a remark. Using Theorem 7.1 of Chapter 6 in \cite{yemelichevKK}, we can show that when $\cT(\br, \bc)$ is a central transportation polytope, the transportation polytopes $\cT(\br(t), \bc(t))$ we defined in Lemma \ref{lem:unipert} achieve the maximum number of vertices among all the transportation polytopes of order $m \times n.$

\section{Central transportation polytope of order $kn \times n$}\label{sec:central}
In this section, %we study the central transportation polytopes of order $kn \times n.$ 
We always assume that $\cT(\br,\bc)$ is a central transportation polytope of order $m \times n$, where $m = kn,$ and $\br = (a, \dots, a)$ and $\bc = (b, \dots, b)$ are two integer vectors. 
We apply the perturbation defined in the last section to family of central transportation polytopes $\{ \cT(\br(t), \bc(t)) \ | \ t \in (0, \frac{1}{m})\}$ satisfying the condition of Theorem \ref{thm:func}, where
\[ \br(t) = (a - t, \dots, a - t), \ \bc(t) = (b, \dots, b, b-mt), \ 0 \le t < \frac{1}{m}.\] 
We have the following theorem on the vertices of $\cT(\br(t),\bc(t)).$ Recall that $\ST_{k,n}$ is the set of spanning trees of $K_{kn,n}$ satisfying $\delta(T) = (k+1, \dots, k+1, k).$

\begin{thm}\label{thm:case2}
	Let $t \in (0, \frac{1}{m}).$ The set of vertices of $\cT(\br(t), \bc(t))$ is in bijection with the set $\ST_{k,n}.$
%the set of the spanning trees $T$ of $K_{m,n}$ satisfying $\delta(T) = (k+1, \dots, k+1, k).$ 
More specifically, $\ST_{k,n}$ is the set of the auxiliary graphs of the vertices of $\cT(\br(t),\bc(t)):$
\[ \ST_{k,n}=\vertAux(\br(t), \bc(t)).\]
Furthermore, for any $T \in \ST_{k,n},$ the corresponding vertex of $\cT(\br(t), \bc(t))$ is the matrix $M_T(t)$ whose entries are defined as below, considering $T$ a tree rooted at $w_n$: 
\begin{equation}\label{equ:case2sol}
	M_T(t)(i,j)  = \begin{cases}a - l(T_{u_i}) \ t, & \text{ if $w_j$ is the parent of $u_i$ in $T$;} \\
		l(T_{w_j}) \ t, & \text{ if $u_i$ is the parent of $w_j$ in $T$;} \\
        0, &  \mbox{otherwise}.\end{cases}
\end{equation}
\end{thm}

%\begin{rem}
%When we consider $T$ as a rooted tree rooted at $w_n,$ having right degree sequence $(k+1,\dots, k+1, k)$ is equivalent to each right vertex $w_j$ have exactly $k$ children.
%\end{rem}

\begin{proof}
We have $akn=bn.$ So $ak = b.$ 

Suppose $M$ is a vertex of $\cT(\br(t), \bc(t)).$ By Theorem \ref{thm:charvert}, the auxiliary graph $T:=\aux(M)$ is a spanning tree. Hence, the number of edges in $T$ is $kn+n-1= (k+1)n-1.$ It is clear that each entry of $M$ cannot exceed $a-t.$ For any $j: 1 \le j \le n-1$, since $\sum_{i=1}^m M(i,j) = b$ and $(a-t)k < b,$ there are at least $k+1$ $i$'s such that $M(i,j)$ is positive. This means the right vertex $w_j$ has at least $k+1$ adjacent edges in $T.$ 
%For the right vertex $w_n,$ since the sum of the weights on its adjacent edges if $b -mt > b-1,$ and $(a-t)(k-1) < b -(a-t) < b-1,$ it has at least adjacent edges. 
Similarly, we can argue that the right vertex $w_n$ has at least $k$ adjacent edges. However, the sum of the right degree sequence of $T$ is equal to the number of edges in $T$ which is $(k+1)n-1$. Thus, we must have $\delta(T) = (k+1, \dots, k+1, k).$

On the other hand, given $T \in \ST_{k,n},$ to show that $T$ is the auxiliary graph of a vertex of $\cT(\br(t),\bc(t)),$ it suffices to verify $M_T(t)$ defined by \eqref{equ:case2sol} is a point in $\cT(\br(t), \bc(t)).$ This can be proved directly by checking the row sum of $M_T(t)$ is always $a-t$, the first $n-1$ column sums are always $b=ak$, and the last column sum is $b-mt.$
\end{proof}

We see that Theorem \ref{thm:mgfcase2} follows from Theorem \ref{thm:case2} and Corollaries \ref{cor:PertAux} and \ref{cor:unipert}.
We can actually analyze the data used to enumerate the formulas in Theorem \ref{thm:mgfcase2} further, which helps to restate Theorem \ref{thm:mgfcase2} with more fundamental combinatorial objects, as well as to figure out the number of vertices in $\cT(\br(t),\bc(t)).$
It is clear that the vertices of $\cT(\br,\bc)$ are the $\{0,a\}$-matrices in which each row has exact one entry of $a$ and each column has exactly $k$ entries of $a.$ These matrices corresponding to ``$k$ to $1$'' matching from the $m=kn$ left vertices to the $n$ right vertices. This motivates the following definition.
\begin{defn}
	We call an $kn \times n$ matrix $M$ a {\it $k$-to-$1$ matching matrix} if $M$ is a $\{0,1\}$-matrix such that there is exactly one $1$ in each row and exactly $k$ $1$'s in each column. We denote by $\Mat_{k,n}$ the set of all the $kn \times n$ $k$-to-$1$ matching matrices. 

	We also call the auxiliary graph of each $k$-to-$1$ matching matrix a {\it $k$-to-$1$ matching graph}.
\end{defn}
With this definition, the vertices of $\cT(\br, \bc)$ is the set 
\[ \vert(\cT(\br,\bc)) = \{ a M \ | \ M \in \Mat_{k,n} \} =: a \Mat_{k,n}.\]

Now we connect the set $\vertAux(\br,\bc)=\ST_{k,n}$ with simpler combinatorial objects,
denoting by $\cR_n$ the set of all the rooted trees on $\{w_1, w_2, \dots, w_n\}$ rooted at $w_n.$ 
\begin{lem}\label{lem:bijrtt}
	There is a bijection between $\Mat_{k,n} \times \cR_n \times [k]^{n-1}$ and $\ST_{k,n}.$
\end{lem}

\begin{proof}
	Given $M \in \Mat_{k,n},$ a rooted tree $R \in \cR_n$ and $\bf = (f_1, \dots, f_{n-1}) \in [k]^{n-1},$ we can construct a tree $T \in \ST_{k,n}$ from $(M, R, \bf)$ in the following  way: We start with the $k$-to-$1$ matching graph $\aux(M).$ If $w_{j_0}$ is the parent of $w_{j}$ in $R,$ we add an edge connecting $w_j$ and the $(f_j)$th left vertex that is matched to $w_{j_0}$ in $\aux(M).$ After adding these $n-1$ edges, one can check that we actually obtain a spanning tree $T$ of $K_{kn,n}$ rooted at $w_n$ and each right vertex $w_j$ has exactly $k$ children. We can ignore the root. Then $T$ is in $\ST_{k,n}.$% in $\PertAux(v).$

	For example, the tree in Figure \ref{fig:exvalue} is the image of the tuple $(M,R,\bf)$ shown in Figure \ref{fig:exR}. %One sees that the solid lines are the edges of the auxiliary graph $\aux(M)$ of $M.$
%\[ M = \begin{pmatrix}1&0&0 \\ 1&0&0 \\ 0&1&0 \\ 0&0&1 \\ 0&1&0 \\ 0&0&1 \end{pmatrix}, R \text{ with edge sets $\{\{w_3,w_1\},\{w_1,w_2\}\}$}, \bf = (1,2). \]
\begin{figure}[htb]
\centering
{\input{exR.pstex_t}}
\caption{}
\label{fig:exR}
\end{figure}

Conversely, let $T \in \ST_{k,n}.$ Considering $T$ a rooted tree rooted at $w_n,$ because $\delta(T) = (k+1, \dots, k+1, k),$ one sees that each right vertex $w_j$ has exactly $k$ children. We separate the edge sets of $T$ into two sets: 
\begin{align*}
	E_1 =& \ \{ e \in E(T) \ | \ \text{$e$ connects a right vertex $w_j$ to one of its children}\}; \\
	E_2 =& \ \{ e \in E(T) \ | \ \text{$e$ connects a right vertex $w_j$ to its parent}\}. 
\end{align*}
In the tree shown in Figure \ref{fig:exvalue}, $E_1$ is the set of solid lines and $E_2$ is the set of the dashed lines.

It is clear that $E_1$ is the edge sets of a $k$-to-$1$ matching graph. Let $M$ be the $k$-to-$1$ matching matrix corresponding to it. If we contract all the edges in $E_1$ in $T$ (and remove all the left vertices $u_i$), we obtain a rooted tree $R \in \cR_n.$ Finally, we define $\bf \in [k]^{n-1}$ from $E_2$: If $\{w_j, u_i\} \in E_2,$ i.e., $u_i$ is the parent of $w_j$ in $T$, and suppose $u_i$ is the $s$th child for its parent, we define $f_j = s.$ Hence, we defined the inverse map.
%
%One sees that the above procedure can be reversed. Thus, we give a bijection between $\PertAux(v)$ and $\cR_n \times [k]^{n-1}.$ 
%
%The cardinality result follows from the famous result that the $|\cR_n| = n^{n-2}.$
\end{proof}

We denote by $\Phi$ the bijection from $\Mat_{k,n} \times \cR_n \times [k]^{n-1}$ to $\ST_{k,n}$ given in the above proof. 
Let $v$ be a vertex of $\cT(\br,\bc),$ i.e., $v = a M$ for some $M \in \Mat_{k,n}.$ Recall that the set $\PertAux(v)$ given in Theorem \ref{thm:mgfcase2} is defined as
\begin{align*}
	\PertAux(v) =& \ \{ T \in \ST_{k,n} \ | \ \text{ the auxiliary graph of $v$ is a subgraph of $T$}\} \\
	=& \ \{ T \in \ST_{k,n} \ | \ \text{ the auxiliary graph of $M$ is a subgraph of $T$}.\} \label{equ:pertauxdes} 
\end{align*}
It is clear that for any fixed vertex $v = a M$ of $\cT(\br,\bc)$, the bijection $\Phi$ induces a bijection from $\cR_n \times [k]^{n-1}$ to $\PertAux(v),$ which maps $(T, \bf)$ to $\Phi(M, T, \bf).$ We denote this map by $\Phi_M.$ Therefore, 
\[ \PertAux(v) = \Phi_M(\cR_n \times [k]^{n-1}).\]

We now restate Theorem \ref{thm:mgfcase2}.
\begin{cor}\label{cor:mgfcase2} Suppose $M \in \Mat_{k,n}.$ Then the MGF of the feasible cone of $\cT(\br,\bc)$ at $v = a M$ is
\begin{equation}\label{equ:mgfcase2}
	 f(\fcone(\cT(\br,\bc), v), \z) =\sum_{T \in \Phi_M(\cR_n \times [k]^{n-1})} \  \prod_{e \not\in E(T)} \frac{1}{1 - \z^{\cyc(T,e)}}.
 \end{equation}
We have two formulas for the MGF of $\cT(\br,\bc)$:
\begin{equation}\label{equ:mgfcase2entire0}
	f(\cT(\br, \bc),\z) =\sum_{M \in \Mat_{k,n}} \z^{a M} \sum_{T \in  \Phi_M(\cR_n \times [k]^{n-1})} \ \prod_{e \not\in E(T)} \frac{1}{1 - \z^{\cyc(T,e)}},
\end{equation}
and
\begin{equation}\label{equ:mgfcase2entire2}
f(\cT(\br, \bc),\z) =\sum_{T \in \ST_{k,n}} \z^{a M_T} \prod_{e \not\in E(T)} \frac{1}{1 - \z^{\cyc(T,e)}},
\end{equation}
where $M_T$ is the matrix whose corresponding graph is the unique $k$-to-$1$ matching subgraph of $T,$ or equivalently, $M_T$ is the $k$-to-$1$ matching matrix entry in the tuple $\Phi^{-1}(T).$
\end{cor}

\begin{proof}
	Formulas \eqref{equ:mgfcase2} and \eqref{equ:mgfcase2entire2} follows from Theorem \ref{thm:mgfcase2} and the discussion above. Formula \eqref{equ:mgfcase2entire0} follows from Formula \eqref{equ:mgfcase2} and Corollary \ref{cor:intepoly}.
\end{proof}

When $k=1$ and $a=1,$ the set $\M_{k,n}$ is actually the symmetric group $\S_n,$ and the polytope $\cT(\br,\bc)$ is the Birkhoff polytope $B_n.$ Therefore, for any $\sigma \in \S_n$, the map $\Phi_\sigma$ is a bijection from $\cR_n$ to $\PertAux(\sigma)$. Hence, we obtain the following theorem for the Birkhoff polytopes, which is equivalent to Theorem 1.1 and Corollary 4.1 in \cite{birkhoff}. 

\begin{thm}\label{thm:birkhoff} Suppose $\sigma \in \S_n$ is a vertex of the Birkhoff polytope $B_n.$ Then the MGF of the feasible cone of $B_n$ at $\sigma$ is
\[ f(\fcone(B_n, \sigma), \z) =\sum_{T \in \Phi_\sigma(\R_n)} \prod_{e \not\in E(T)} \frac{1}{1 - \z^{\cyc(T,e)}}.\] 
We have two formulas for the MGF of $B_n$:
	\[f(B_n,\z) =\sum_{\sigma \in \S_n} \z^{\sigma}\sum_{T \in \Phi_\sigma(\R_n)} \prod_{e \not\in E(T)} \frac{1}{1 - \z^{\cyc(T,e)}},\]
	and
	\[f(B_n,\z)= \sum_{T \in \ST_{1,n}} \z^{M_T} \prod_{e \not\in E(T)} \frac{1}{1 - \z^{\cyc(T,e)}},\]
where $M_T$ is the matrix whose corresponding graph is the unique $1$-to-$1$ matching subgraph of $T.$ 
\end{thm}

Next, we give formulas for the number of vertices of $\cT(\br(t),\bc(t)).$
\begin{cor}
The number of vertices of $\cT(\br(0),\br(0)) = \cT(\br,\bc)$ is $\displaystyle \frac{(kn)!}{(k!)^n}.$

The number of vertices of $\cT(\br(t),\bc(t))$ ($t \in (0, \frac{1}{m})$) is $\displaystyle \frac{(kn)!}{(k!)^n} n^{n-2} k^{n-1}.$
\end{cor}

\begin{proof}
	The first statement follows from the observation that the cardinality of $\Mat_{k,n}$ is $\displaystyle \frac{(kn)!}{(k!)^n}.$

	It is well-known that the number of rooted trees on $n$ vertices with a fixed root is $n^{n-2}.$ Then the second statement follows from Lemma \ref{lem:bijrtt}.
\end{proof}

%By the remark we give at the end of last section, we obtain another known result.
By Lemma \ref{lem:maxvert}, we obtain another known result.
\begin{cor}[Corollary 8.6 of Chapter 6 in \cite{yemelichevKK}]
The maximum number of vertices among all the transportation polytopes of order $kn \times n$ is $\displaystyle \frac{(kn)!}{(k!)^n} n^{n-2} k^{n-1}.$
\end{cor}

Finally, we apply Lemma \ref{lem:mgf2volehr} to Formula \eqref{equ:mgfcase2entire2} to obtain formulas for the volume and Ehrhart polynomial of the central transportation polytope of order $kn \times n.$

\begin{cor}
	Let $C$ be a $kn \times n$ matrix such that $\langle C, \cyc(T,e)\rangle \neq 0$ for any pair of $T \in \ST_{k,n}$ and $e \not\in E(T).$
%Then the Ehrhart polynomial of the central transportation polytope $\cT(\br,\bc)$ of order $kn \times n$ is given by
%\[ i(\cT(\br,\bc),t) = \sum_{i=0}^{(kn-1)(n-1)} \frac{t^i}{i!} \sum_{T \in ST_{k,n}}  \frac{\left( \langle c, aM_T \rangle \right)^{i} \td_{(kn-1)(n-1)-i}\left( \langle c, \cyc(T,e) \ | \ e \not\in E(T) \right)}{\prod_{e \not\in E(T)} \langle c, \cyc(T,e) \rangle } .\]
Then the coefficient of $t^i$ in the Ehrhart polynomial of the central transportation polytope $\cT(\br,\bc)$ of order $kn \times n$ is given by
\[ \frac{1}{i!} \sum_{T \in ST_{k,n}}  \frac{\left( \langle C, aM_T \rangle \right)^{i} \td_{(kn-1)(n-1)-i}\left( \langle C, \cyc(T,e) \rangle \ | \ e \not\in E(T) \right)}{\prod_{e \not\in E(T)} \langle C, \cyc(T,e) \rangle } .\]
In particular, the normalized volume of $\cT(\br,\bc)$ is given by
\[ \vol(\cT(\br,\bc)) = \frac{1}{((kn-1)(n-1))!} \sum_{T \in ST_{k,n}}  \frac{\left( \langle C, aM_T \rangle \right)^{(kn-1)(n-1)}}{\prod_{e \not\in E(T)} \langle C, \cyc(T,e) \rangle } .\]

\end{cor}

We remark that when $k=1,$ the above corollary gives an equivalent result to Corollary 1.2 in \cite{birkhoff}.

\subsection*{Acknowledgements}
I would like to thank Bernd Sturmfels for encouraging me to work on this project.

\bibliographystyle{amsplain}
\bibliography{gen}

\end{document}